\documentclass[a4paper,11pt, reqno]{amsart}
\usepackage[utf8]{inputenc}
\usepackage{amsmath}
\usepackage{amssymb}
\usepackage{amsthm}
\usepackage{amsrefs} 
\usepackage{a4wide}
\usepackage{paralist}
\usepackage{marvosym}
\usepackage{esint}
\usepackage{dsfont}

\newtheorem{theorem}{Theorem}[section]
\newtheorem{lemma}{Lemma}[section]

\newtheorem{corollary}{Corollary}[section]
\newtheorem{example}{Example}[section]

\theoremstyle{definition}

\theoremstyle{remark}
\newtheorem{remark}{Remark}[section]

\DeclareMathOperator{\E}{\mathds{E}}

\DeclareMathOperator{\Var}{Var}
\DeclareMathOperator{\N}{\mathbb{N}}
\DeclareMathOperator{\Z}{\mathbb{Z}}
\DeclareMathOperator{\R}{\mathbb{R}}

\DeclareMathOperator{\1}{\mathds{1}}

\newcommand{\lnorm}[3]{\left\Vert{#3}\right\Vert_{\ell^{#1}(\N)^{\otimes{#2}}}}
\newcommand{\lnormb}[3]{\big\Vert{#3}\big\Vert_{\ell^{#1}(\N)^{\otimes{#2}}}}

\title[Limit theorems for Rademacher functionals]{Berry-Esseen bounds and multivariate limit theorems\\ for functionals of Rademacher sequences}

\thanks{The authors have been supported by the German Research Foundation (DFG) via SFB-TR 12 ``Symmetries and Universality in Mesoscopic Systems''.}

\author[K. Krokowski]{Kai Krokowski}
\address{Kai Krokowski: Faculty of Mathematics, NA 3/28, Ruhr University Bochum, Germany}
\email{kai.krokowski@rub.de}

\author[A. Reichenbachs]{Anselm Reichenbachs}
\address{Anselm Reichenbachs: Faculty of Mathematics, NA 3/28, Ruhr University Bochum, Germany}
\email{anselm.reichenbachs@rub.de}

\author[C. Th\"ale]{Christoph Th\"ale}
\address{Christoph Th\"ale: Faculty of Mathematics, NA 3/68, Ruhr University Bochum, Germany}
\email{christoph.thaele@rub.de}


\begin{document}

\begin{abstract}
Berry-Esseen bounds for non-linear functionals of infinite Rademacher sequences are derived by means of the Malliavin-Stein method. Moreover, multivariate extensions for vectors of Rademacher functionals are shown. The results establish a connection to small ball probabilities and shed new light onto the relation between central limit theorems on the Rademacher chaos and norms of contraction operators. Applications concern infinite weighted 2-runs, a combinatorial central limit theorem and traces of Bernoulli random matrices.
\end{abstract}

\maketitle

\section{Introduction}

In the seminal paper \cite{NouPecPTRF09} by Nourdin and Peccati it has been demonstrated for the first time that there is a powerful connection between Stein's method and the Malliavin calculus of variations. A main feature of this approach is that the various coupling constructions, which are usually in the background of Stein's method, are replaced by a structural property of the involved random variables. In practice, this is reflected by the use of Malliavin operators and especially through application of an integration-by-parts formula. We mention one of the main results of the paper \cite{NuaPec} of Nualart and Peccati, which shows that a sequence of Gaussian multiple integrals $F_n:=I_q(f_n)$ of fixed order $q\geq 1$ of symmetric and square-integrable functions $f_n$ such that $\E[F_n^2]=1$ for all $n\geq 1$ converges to a standard Gaussian random variable $N$ if and only if the 4th cumulant
$$\kappa_4(F_n):=\E[F_n^4]-3$$
converges to $0$, as $n\to\infty$. Using Stein's method, this has been extended to an estimate on various probability distances between $F_n$ and $N$ (cf. \cites{NouPecPTRF09,NouPecReiInvariance,NouPecBook}).
In fact, a combination of the main results of \cite{NouPecPTRF09} and \cite{NouPecReiInvariance} (see also Theorem 5.2.6 in \cite{NouPecBook}) asserts that
\begin{equation}\label{eq:4thMomentBoundGaussianCase}
d_{K}(F_n,N):=\sup_{x\in\R}\big|P(F_n\leq x)-P(N\leq x)\big|\leq C\,\kappa_4(F_n)^{1/2}
\end{equation}
with a constant $0<C<\infty$ not depending on $n$.  A key step in the proof of this bound is an estimate of $d_K(F_n,N)$ in terms of contraction operators $f_n\star_r^r f_n$, and a comparison with an expression for the 4th cumulant $\kappa_4(F_n)$ in terms of such contractions. We emphasize that the quantitative 4th moment theorem \eqref{eq:4thMomentBoundGaussianCase} is one of the cornerstones and at the heart of the Malliavin-Stein approach. The following optimal bound on the total variation distance $d_{TV}(F_n,N):=\sup\big\{\big|P(F_n\in A)-P(N\in A)\big|:A\subset\R\ \text{Borel}\big\}$ has recently been derived in \cite{NouPecOpt}. Namely, if $F_n$ converges in distribution to a standard Gaussian random variable $N$, then
\begin{equation}\label{eq:4thMomentBoundGaussianCase2}
C_1\max\{|\E[F_n^3]|,\kappa_4(F_n)\}\leq d_{TV}(F_n,N)\leq
C_2 \max\{|\E[F_n^3]|,\kappa_4(F_n)\}\,,\end{equation}
where $0<C_1<C_2<\infty$ are constants which are independent of $n$. Since $d_K(F_n,N)\leq d_{TV}(F_n,N)$, this improves \eqref{eq:4thMomentBoundGaussianCase} significantly in case that $\E[F_n^3]=0$, for example, if $q$ is odd. Since its first appearance, the 4th moment theorem has attracted considerable interest and has further been exploited by many authors. Selected applications concern central limit theorems for non-linear functionals of Gaussian stochastic processes \cite{NouPecBerryEsseen}, random fields on the sphere \cite{MarPec}, random matrices \cite{NouPecMatrices} and universality of homogeneous sums \cite{NouPecReiInvariance} (we also refer to the monograph \cite{NouPecBook} and the exhaustive list of references therein).

\medspace

Besides non-linear functionals of Gaussian random measures, there is another branch to which the Malliavin-Stein approach has been applied, namely non-linear functionals of Poisson random measures, see the papers \cite{PecSolTaqUtz,PecTha,PecZhe} of Peccati, Sol\'e, Taqqu, Utzet and Zheng. These general results have found numerous applications especially in geometric probability and stochastic geometry, see \cite{LacPec1,LacPec2,LPST,ReiSch,Sch} for distinguished examples developed by Lachi\`eze-Rey, Last, Peccati, Penrose, Reitzner, Schulte and Th\"ale. We emphasize that it has been shown by Eichelsbacher and Th\"ale \cite{EicTha}, and Lachi\`eze-Rey and Peccati \cite{LacPec1}, who combined Stein's method with the Malliavin calculus of variations on the Poisson space, that a quantitative 4th moment theorem similar to \eqref{eq:4thMomentBoundGaussianCase} is also available for multiple stochastic integrals with respect to Poisson random measures if the integrands $f_n$ are non-negative.

\medspace

In \cite{ReiPec}, Nourdin, Peccati and Reinert pushed this line of research further by proving quantitative central limit theorems for functionals of so-called infinite Rademacher sequences, which rely on a combination of Stein's method with tools from discrete stochastic analysis as developed in \cite{Pri} by Privault. By a Rademacher sequence we mean in this paper an infinite sequence $X=(X_n)_{n\in\N}$ of independent and identically distributed random variables such that $X_n$ takes the values $\pm 1$ with probability $1/2$. The aim of this paper is to develop the theory of \cite{ReiPec} further in several directions. In particular, our main findings are

\begin{itemize}
\item[(i)] an estimate for the Kolmogorov distance between a possibly non-linear functional of a Radema\-cher sequence $X$ (this is what we call a Rademacher functional) and a Gaussian random variable in terms of Malliavin operators. This refines the bounds of \cite{ReiPec}, where only smooth distances have been considered. The proof of this bound is a non-trivial task as it relies on various new computations involving discrete Malliavin operators and also on a new integration-by-parts formula.
\item[(ii)] a connection between our Malliavin-Stein bound and quantities, which are known as small ball probabilities. They are a measure of anti-concentration and enter the expression for the Kolmogorov distance. We mention that they were not visible in the previous work \cite{ReiPec} because of the smoothness of the test functions used there. It is worth pointing out that small ball probabilities in the context of Berry-Esseen bounds have previously found attention in the works \cite{LitPajRud,RudVer} of Litvak, Pajor Rudelson and Tomczak-Jaegermann, and Rudelson and Vershynin, respectively.
\item[(iii)] a quantitative multivariate central limit theorem dealing with the distance between a vector of Rademacher functionals and a Gaussian random vector. In particular, we show that a vector consisting of discrete multiple stochastic integrals satisfies a multivariate central limit theorem if its entries fulfil univariate central limit theorems. This is the discrete analogue to a similar phenomenon observed by Nourdin, Peccati and R\'eveillac \cite{NouPecRev}, Nualart and Ortiz-Latorre \cite{NuaOrt}, Peccati and Tudor \cite{PecTud} and Peccati and Zheng \cite{PecZhe} for Gaussian or Poisson multiple integrals, respectively.
\item[(iv)] a clarification of the r\^ole of contraction operators in light of necessary conditions for a sequence of discrete multiple stochastic integrals to satisfy a central limit theorem. We develop a new necessary criterion for discrete double integrals, which is based on a novel representation of the 4th cumulant involving on- and off-diagonal terms. We also present a counterexample showing that vanishing contraction norms do not provide a necessary condition. This sheds new light onto a 4th moment theorem for Rademacher functionals, and general quadratic forms as considered by de Jong \cite{Jon} or Chatterjee \cite{Cha}, for example.
\item[(v)] applications of our results to infinite weighted 2-runs, an extended version of a combinatorial central limit theorem and traces of powers of Bernoulli random matrices. This extends the previous results from \cite{BleJan} and also some of the findings in \cite{ReiPec,NouPecMatrices}. In particular, we provide a Berry-Esseen bound for an infinite-dimensional version of a central limit theorem of Blei and Janson and give a direct proof for a multivariate central limit theorem for traces of powers of Bernoulli matrices without resorting to universality results.
\end{itemize}
Our results rely on Stein's method for normal approximation and tools from discrete stochastic analysis. We recall both together with some other preliminaries in Section \ref{sec:Preliminaries}. There, we also provide a new integration-by-parts formula on which our proofs are based on. The one-dimensional Malliavin-Stein bound is the content of Section \ref{sec:1dimbound}, while Section \ref{sec:multipleintegrals} discusses various versions in case of sequences of discrete multiple stochastic integrals. In this context we also develop the announced new necessary criterion for a sequence of discrete stochastic double integrals to satisfy a central limit theorem. Multivariate central limit theorems for vectors of Rademacher functionals are provided in Section \ref{sec:multivariate}. The final Section \ref{sec:applications} contains the applications of our results.

\section{Preliminaries}\label{sec:Preliminaries}

\subsection{Rademacher sequences.}
By a {Rademacher sequence} $(X_n)_{n \in \N}$ we understand a sequence consisting of i.i.d.\ random variables $X_n$ defined on some probability space $(\Omega,\mathcal{F},P)$ such that $$P(X_n=-1)=P(X_n=+1)=\frac{1}{2}\,.$$ They are constructed in the canonical way, namely by taking $$\Omega:=\{-1,+1\}^{\N}\,,\qquad\mathcal{F}:=\mathcal{P}(\{-1,+1\})^{\otimes\N}\,,\qquad P:=\Big(\frac{1}{2}\delta_{-1}+\frac{1}{2}\delta_{+1}\Big)^{\otimes\N}\,,$$ where $\delta_{\pm 1}$ is the unit-mass Dirac measure concentrated at $\pm 1$, and then putting $X_n(\omega):=\omega_n$ for $(\omega_n)_{n\in\N}\in\Omega$. Here and below, we write $\mathcal{P}(M)$ for the power set of a set $M$.

\subsection{Kernels and contractions.}
In what follows we will denote by $\kappa$ the counting measure on $\N$. For $n \geq 1$ define $\ell^2(\N)^{\otimes n} := L^2(\N^n,\mathcal{P}(\N)^{\otimes n},\kappa^{\otimes n})$. Functions in $\ell^2(\N)^{\otimes n}$ are called {kernels} in the sequel. The following subsets of $\ell^2(\N)^{\otimes n}$ are of interest. Let $\ell^2(\N)^{\circ n}$ denote the class of symmetric kernels and $\ell_0^2(\N)^{\otimes n}$ denote the class of kernels vanishing on the diagonal, i.e., on the complement $\Delta_n^c$ of the set 
\[\Delta_n:=\{(i_1,\dots,i_n)\in\N^n: i_k\neq i_l \text{ for } k\neq l\}.\]
Let $\ell_0^2(\N)^{\circ n}$ denote the class of symmetric kernels vanishing on diagonals. For $q=0$ one defines $\ell^2(\N)^{\circ 0}:=\R$. For integers $n, m \geq 1$, $r\in\{0,\dots,n \wedge m\}$, $l\in\{0,\dots,r\}$ and kernels $f\in\ell^2_0(\N)^{\circ n}$ and $g\in\ell^2_0(\N)^{\circ m}$ the {contraction}
\begin{align*}
&f\star_r^lg(i_1,\dots,i_{n-r},j_{1},\dots,j_{r-l},k_{1},\dots,k_{m-r})\\
&:=\sum_{(a_1,\dots,a_l)\in\Delta_l}f(i_1,\dots,i_{n-r},j_{1},\dots,j_{r-l},a_1,\dots,a_l)\, g(k_1,\dots,k_{m-r},j_{1},\dots,j_{r-l},a_1,\dots,a_l)
\end{align*}
arises from the tensor product of $f$ and $g$ by first identifying $r$ of the $n+m$ variables and then by integrating out $l$ of them with respect to the counting measure $\kappa$. In particular, for $f\in\ell^2_0(\N)^{\circ n}$ the contraction $f\star^0_0f$ is the tensor product of $f$ with itself. 
For a function $f$ on $\N^n$ denote by $\tilde{f}=\frac{1}{n!}\sum_{\sigma}f(i_{\sigma(1)},\dots,i_{\sigma(n)})$ its canonical symmetrization, where the sum runs over all permutations $\sigma$ of $\{1,\dots,n\}$. Since $f\star_r^lg$ is usually not symmetric, we often consider its canonical symmetrization $\widetilde{f\star_r^lg}$. Note that $\|\widetilde{f}\|_{\ell^2(\N)^{\otimes n}}\leq\|{f}\|_{\ell^2(\N)^{\otimes n}}$ for any $f\in\ell^2(\N)^{\otimes n}$.

\begin{lemma} Let $q \geq 2$ and suppose that $f\in\ell_0^2(\N)^{\circ q}$.
\begin{itemize}
\item[(i)] It holds that
\begin{align}
(2q)! \, \big\| \widetilde{f \star_0^0 f} \big\|_{\ell^2(\N)^{\otimes 2q}}^2 = 2 \big( q! \lnorm{2}{q}{f}^2 \big)^2 + \sum_{r=1}^{q-1}  (q!)^2\binom{q}{r}^2 \lnorm{2}{2(q-r)}{f \star_r^r f}^2\,. \label{Taqqu}
\end{align}
\item[(ii)] One has
\begin{align}\label{eq:Relation01}
\lnorm{2}{2q-1}{f \star_1^0 f} = \big\| f \star_q^{q-1} f \big\|_{\ell^2(\N)}
\end{align}
and
\begin{align}
\lnorm{2}{2(q-r)+1}{f \star_r^{r-1} f} \leq \lnorm{2}{2(q-r+1)}{f \star_{r-1}^{r-1} f} \label{Contraction inequality}
\end{align}
 for every $r \in\{2,\ldots,q\}$.
\end{itemize}
\end{lemma}
\begin{proof}
Identity \eqref{Taqqu} is Formula (11.6.30) in \cite{PecTaq} and for part (ii) we refer to Lemma 2.4 in \cite{ReiPec}.
\end{proof}

\subsection{Multiple stochastic integrals and chaotic decomposition.}
For $q \geq 1$ and $f\in\ell^2_0(\N)^{\circ q}$ the {discrete multiple stochastic integral of order $q$} of $f$ is defined as 
\begin{equation}
\label{multint}
J_q(f):=q!\sum_{1\leq i_1<\dots<i_q<\infty}f(i_1,\dots,i_q)\,X_{i_1}\cdots X_{i_q}\,.
\end{equation}
For $q=0$ and $c\in\R$ we put $J_0(c):=c$. The family of random variables of the form $J_q(f)$ with $f\in\ell^2_0(\N)^{\circ q}$ is called {Rademacher chaos of order $q$} (sometimes also called the Walsh chaos of order $q$). Multiple stochastic integrals fulfil the isometry relation
\begin{equation}\label{Isometric relation}
\E \left[ J_q(f)J_p(g) \right] = \1_{\{q=p\}}q!\langle f,g \rangle_{\ell^2(\N)^{\otimes q}}\,.
\end{equation}
Moreover, it is a crucial fact that every $F\in L^2(\Omega)$ possesses a unique decomposition in terms of multiple stochastic integrals, i.e., each $F\in L^2(\Omega)$ can be written as
\[F=\E(F)+\sum_{n=1}^{\infty}J_n(f_n)\,,\]
where $f_n\in\ell^2_0(\N)^{\circ q},n\geq 1,$ is a uniquely determined sequence of kernels (see Section 6 in \cite{Pri}).
In our paper, the following multiplication formula for discrete stochastic integrals will turn out to be crucial (see Proposition 2.9 in \cite{ReiPec}). It says that for all integers $p, q \geq 1$ and symmetric kernels $f \in \ell^2_0(\N)^{\circ q}$ and $g \in \ell^2_0(\N)^{\circ p}$, it holds that
\begin{equation}\label{Multiplication formula}
J_q(f)J_p(g) = \sum_{r=0}^{q \wedge p} r! \binom{q}{r} \binom{p}{r} J_{q+p-2r}\big( ( \widetilde{f \star_r^r g} ) \1_{\Delta_{q+p-2r}}\big)\,. 
\end{equation}

\subsection{Discrete Malliavin operators.}
We are now going to introduce the four important so-called Malliavin operators for which we refer to \cite{Pri}.
\subsubsection{Gradient operator.} 
The {gradient operator} $D$ transforms random variables into random sequences and is defined as $$DJ_q(f_q):=(D_kJ_q(f_q))_{k\in\N}=(J_{q-1}(f_q(\,\cdot\,,k)))_{k\in\N}\in L^2(\Omega\times\N,P\otimes\kappa)$$ on a Rademacher chaos of fixed order $q\geq 1$. We also put $D_kJ_0(c)=0$ for $c\in\R$ and $k\in\N$. It can consistently be extended to the class of functionals $F\in L^2(\Omega)$ of the form $F=\E(F)+\sum_{n=1}^{\infty}J_n(f_n)$, which satisfy the relation
\begin{equation}\label{domD}
\E[\|DF\|^2_{\ell^2(\N)}]=\sum_{n=1}^{\infty}n\,n!\lnorm{2}{n}{f_n}^2<\infty\,.
\end{equation}
The class of all such functionals is the domain of $D$ and will be denoted by dom($D$).
For $F\in\text{dom}(D)$ one has
\[DF=(D_kF)_{k\in\N}=\Big(\sum_{n=1}^{\infty}nJ_{n-1}(f_n(\cdot,k))\Big)_{k\in\N}.\]
The operator $D$ also admits a pathwise representation, which can be used as an alternative definition, especially if the functional does not satisfy condition \eqref{domD}. To spell this out, define for $\omega=(\omega_n)_{n\in\N}\in\Omega$,
\[ \omega_k^+ := (\omega_1,\dots,\omega_{k-1},+1,\omega_{k+1},\dots) \quad\text{and}\quad \omega_k^- := (\omega_1,\dots,\omega_{k-1},-1,\omega_{k+1},\dots)\,. \]
We further define for $F\in L^2(\Omega)$ and $\omega\in\Omega$ the functionals $F_k^+$ and $F_k^-$ by $$F_k^+(\omega):=F(\omega_k^+)\qquad\text{and}\quad F_k^-(\omega):=F(\omega_k^-)\,.$$
Then, for $F\in {\rm dom}(D)$ one has the relation
\[D_kF(\omega)=\frac{1}{2}(F(\omega_k^+)-F(\omega_k^-))\,.\]
Formally, let us denote by $D'$ the operator acting on $F\in L^2(\Omega)$ by 
\[D'_kF(\omega):=\frac{1}{2}(F(\omega_k^+)-F(\omega_k^-))\,.\]
We shall discuss the difference between $D$ and $D'$ now, but before, we define inductively the higher-order gradients by putting $D^n_{k_1,\ldots,k_n}F:=D_{k_1}D^{n-1}_{k_2,\ldots,k_n}F$ for $n\geq 1$, where $D^0:=\text{Id}$ and $D^1:=D$,  and similarly for $(D')^n_{k_1,\ldots,k_n}F$.


\begin{lemma}\label{lem:StroockAndOthers}
\begin{enumerate}[(i)]
\item If $F=\sum_{n=0}^\infty J_n(f_n)$ with $f_n\in\ell^2_0(\N)^{\circ n}$ then 
	\[\E[F\cdot X_{k_1}\cdots X_{k_n}]=n! f_n(k_1,\dots,k_n)\quad\text{for all}\quad(k_1,\ldots,k_n)\in\Delta_n.\]
\item For $F\in L^2(\Omega)$ it holds that
	\begin{equation}\label{stroock}
	\E[(D')^n_{k_1,\ldots,k_n}F]=\E[F\cdot X_{k_1}\cdots X_{k_n}]\quad \text{for all}\quad (k_1,
	\ldots,k_n)\in\Delta_n\,.
	\end{equation}
\item For $F,G\in L^2(\Omega)$ and $k\in\N$ one has
	\begin{equation}\label{Product formula gradient operator}
	D_k(FG) = GD_kF + FD_kG - 2X_k(D_kF)(D_kG).
	\end{equation}
\end{enumerate}
\end{lemma}
\begin{remark}
Lemma \ref{lem:StroockAndOthers} (i) and (ii) show that
\begin{equation}\label{eq:StrookRichtig}
f_n(k_1,\dots,k_n)={1\over n!}\E[(D')^n_{k_1,\ldots,k_n}F]
\end{equation}
if $F$ has chaotic decomposition $F=\sum_{n=0}^\infty J_n(f_n)$. This is the analogue of the classical {Stroock formula} for Rademacher functionals. 
\end{remark}
\begin{proof}[Proof of Lemma \ref{lem:StroockAndOthers}]
Part (i) follows directly from the definition \eqref{multint} of a discrete multiple stochastic integral. To prove (ii) we first observe that for every $k\in\N$, 
\[\E[F_k^+]=\E[X_k\cdot F]+\E[F]\quad{\rm and}\quad\E[F_k^-]=\E[-X_k\cdot F]+\E[F].\]
We thus get for $k\in\N$ that
\begin{align*}
\E[D'_kF]&=\frac{1}{2}\E[(F_k^+-F_k^-)]=\frac{1}{2}\big((\E[X_k\cdot F]+\E[F])-(\E[-X_k\cdot F]+\E[F])\big)\\
&=\E[X_k\cdot F]\,.
\end{align*}
The general case is proved by induction. Let $n\in\N$ and $(k_1,\ldots,k_{n+1})\in\Delta_{n+1}$. Then, using the induction hypothesis and the fact that from the point of view of $D_{k_1,\ldots,k_n}$ the $k_{n+1}$st entry $X_{k_{n+1}}$ of the Rademacher sequence $X$ behaves like a constant, we find that
\begin{align*}
\E[(D')_{k_1,\ldots,k_{n+1}}^{n+1}F] &= \E[D'_{k_{n+1}}((D')_{k_1,\ldots,k_n}^nF)] = \E[X_{k_{n+1}}\cdot (D')^n_{k_1,\ldots,k_n}F]\\
&= \E[(D')^n_{k_1,\ldots,k_n}(F\cdot X_{k_{n+1}})] = \E[(F\cdot X_{k_{n+1}})\cdot X_{k_1}\cdots X_{k_n}]\\
&=\E[F\cdot X_{k_1}\cdots X_{k_{n+1}}]\,.
\end{align*}
This proves assertion (ii). Part (iii) corresponds to Proposition 7.8 in \cite{Pri}.
\end{proof}

The next lemma formalizes Remark 2.11 in \cite{ReiPec}.
\begin{lemma}
\label{DandDprime}
Let $F\in L^2(\Omega)$. Then $\E\big[\sum_{k=1}^{\infty}(D'_kF)^2\big]<\infty$, if and only if  $F\in\text{dom}(D)$.
\end{lemma}
\begin{proof}
Let $F=\E[F]+\sum_{k=1}^\infty J_n(f_n)$ be the chaotic decomposition of the square integrable random variable $F$ for a sequence of kernels $f_n\in\ell_0^2(\N)^{\circ n}$.
The condition $\E\left[\sum_{k=1}^{\infty}(D'_kF)^2\right]<\infty$ implies that $D'_kF\in L^2(\Omega)$ for all $k\in\N$. Therefore $D'_kF$ has a chaotic decomposition of the form
\[D'_kF=\sum_{n=0}^\infty J_n(f'_n)\]
for a sequence of kernels $f'_n\in\ell_0^2(\N)^{\circ n}$. The Stroock formula \eqref{eq:StrookRichtig} yields
\begin{align*}
f'_n(k_1,\dots,k_n)&=\frac{1}{n!}\E[(D')^n_{k_1,\dots,k_n}D'_kF]=\frac{(n+1)!}{n!}f_{n+1}(k_1,\dots,k_n,k)\\
&=(n+1)f_{n+1}(k_1,\dots,k_n,k)\,.
\end{align*}
We thus get the representation 
\begin{align*}
D'_kF&=\sum_{n=0}^\infty (n+1)\, J_{n}(f_{n+1}(\,\cdot\,, k))=\sum_{n=1}^\infty n\, J_{n-1}(f_{n}(\,\cdot\,, k))
\end{align*}
for the chaotic decomposition of $D'_kF$ which is equal to $D_kF$ and implies that $F\in\text{dom}(D)$. This completes the proof.
\end{proof}

\subsubsection{Divergence operator.} 
The {divergence operator} is the adjoint of $D$. It is given by 
\begin{equation*}
\delta(J_n(u_{n+1}(\,*\,,\,\cdot\,))):=J_{n+1}(\widetilde{u_{n+1}})\,,\qquad n\in\{0,1,\ldots\}\,,
\end{equation*}
if $u_{n+1}\in\ell^2(\N)^{\circ n}\otimes\ell^2(\N)$. It can consistently be extended to the class of random functions $u(\,*\,,\,\cdot\,)\in L^2(\Omega\times\N,P\otimes\kappa)$ with $u(\,*\,,k)=\sum_{n=0}^\infty J_n(u_{n+1}(\,*\,,k))$ and kernels $u_{n+1}(\,*\,,k)\in\ell_0^2(\N)^{\circ n}$ satisfying
\begin{equation*}
\E[\delta(u)^2]=\sum_{n=0}^\infty (n+1)!\|\widetilde{ u_{n+1}}\|_{\ell^2(\N)^{\otimes(n+1)}}^2<\infty\,.
\end{equation*}
The class of these random functions is called the domain of $\delta$ and is denoted by dom($\delta$).
Similar to the difference operator, also the divergence operator admits a pathwise representation, namely
\begin{equation}
\label{deltapathwise}
\delta(u)=\sum_{k=1}^\infty u_kX_k-\sum_{k=1}^\infty D_k u_k
\end{equation}
for $u\in\text{dom}(\delta)$, see \cite[Proposition 9.3]{Pri}.

\subsubsection{Ornstein-Uhlenbeck operator and its inverse.} The {Ornstein-Uhlenbeck operator} $L$ is defined by the relation 
\begin{equation}\label{eq:L=-dD}
L:=-\delta D
\end{equation}
for elements of a fixed Rademacher chaos, see \cite[Proposition 10.1]{Pri}. In other words this means that $$L J_n(f_n)=-nJ_n(f_n)$$ for $f_n\in\ell_0^2(\N)^{\circ n}$. The domain of $L$ is the class of all functionals $F=\sum_{n=0}^\infty J_n(f_n)\in L^2(\Omega)$ such that $\E[(LF)^2]=\sum_{n=1}^\infty n^2n!\lnorm{2}{n}{f_n}^2<\infty$. We notice that $L$ maps $F$ into the class of square-integrable centred random variables $L^2_0(\Omega)$. The (pseudo-) inverse operator $L^{-1}$ of $L$ is defined on $L^2_0(\Omega)$ by $$L^{-1}F=-\sum_{n=1}^\infty\frac{1}{n}J_n(f_n)$$ if $F\in L^2_0(\Omega)$ has representation $F=\sum_{n=1}^\infty J_n(f_n)$.

\subsection{Important identities.}
The following lemma collects two important identities, namely the integration-by-parts formula and an isometric formula for the divergence operator.


\begin{lemma} 
\begin{enumerate}[(i)]
\item Let $F\in \text{dom}(D)$ and $u\in\text{dom}(\delta)$, then
	\begin{equation}\label{intbyparts}
	\E[F\delta(u)]=\E[\langle DF,u\rangle_{\ell^2(\N)}]\,.
	\end{equation}
\item For all $u\in\text{dom}(\delta)$ it holds that
	\begin{equation}\label{Isometry property}
	\E [\delta(u)^2]=\E[\|u\|_{\ell^2(\N)}^2]+\E\Big[\sum_{k,l=1}^\infty 
	(D_ku_l)(D_lu_k)\Big]\,.
	\end{equation}
\end{enumerate}
\end{lemma}
\begin{proof}
Part (i) is \cite[Proposition 9.2]{Pri} and for part (ii) we refer to \cite[Proposition 9.3]{Pri}.
\end{proof}


In the proof of one of our main results we will encounter the expression $D'\1_{\{F>x\}}$ with $x\in\R$ and would like to apply the integration-by-parts formula \eqref{intbyparts} to it. Unfortunately, it is not clear in general whether the integrability condition of Lemma \ref{DandDprime} is satisfied for $\1_{\{F>x\}}$ or not. To overcome this difficulty we follow the strategy introduced in \cite{Sch1} in a different context and now develop an integration-by-parts formula for functionals $F$ not necessarily belonging to dom($D$). 

\begin{lemma}\label{Integration by parts II}
Suppose that $F\in L^2(\Omega)$ is bounded, $u\in L^2(\Omega\times\N)$ is such that $u_k:=u(\,\cdot\,,k)$ is independent of $X_k$ and $(D'_k F)u_k\geq 0$ for all $k\in\N$. Then,
\[\E[F\delta(u)]=\E[\langle D'F,u\rangle_{\ell^2(\N)}]\,.\]
\end{lemma}
\begin{proof}
First observe that $\delta(u)=\sum_{k=1}^\infty u_k X_k$ by \eqref{deltapathwise}, due to the independence assumption. Using \eqref{stroock}, we thus find that
\begin{align*}
\E[\langle D'F,u\rangle]&=\E\Big[\sum_{k=1}^\infty (D'_kF)u_k\Big]=\sum_{k=1}^\infty \E[D'_k(F\cdot u_k)]\\
&=\sum_{k=1}^\infty\E[(F\cdot u_k)\cdot X_k]\\
&=\E[F\cdot\delta(u)]\,.
\end{align*}
The exchange of the order of summation is justified by Fubini's theorem, since $(D'_k F)u(k)\geq 0$ by assumption and since
\[\E\Big[\sum_{k=1}^\infty|F\cdot u_k X_k|\Big]\leq C\,\E\Big[\sum_{k=1}^\infty |u_k|\Big]<\infty\]
for a suitable constant $0<C<\infty$ by boundedness of $F$ and square integrability of $u$.
\end{proof}

From now on and to simplify the notation we will write $DF$ for the discrete gradient applied to a Rademacher functional $F\in L^2(\Omega)$ and interpret this as $D'F$ if $F\in L^2(\Omega)\setminus{\rm dom}(D)$.

\subsection{Stein's method for one-dimensional normal approximation.}
It is a well known fact that a real-valued random variable $Z$ follows a standard normal (or Gaussian) distribution if and only if $$\E[f'(Z)-Zf(Z)]=0$$ for all bounded, continuous and piecewise continuously differentiable functions $f:\R\to\R$ satisfying $\E|f'(Z)|<\infty$, see \cite[Lemma 2.1]{CheGolSha}. This is the so-called Stein-type characterization of the standard normal distribution and the corresponding Stein-equation reads
\begin{equation}\label{steineq}
f'(z)-zf(z)=\1(z\leq x) - P(N\leq x)\,, \quad x\in\R\,,
\end{equation}
where $N$ stands for a standard normal random variable. For a given $x$, a solution of \eqref{steineq} will be denoted by $f_x(z)$. Taking expectations in \eqref{steineq} suggests to re-write the Kolmogorov distance
$$d_K(Z,N):=\sup_{x\in\R}\big|P(Z\leq x)-P(N\leq x)\big|$$
between (the distributions of) $Z$ and $N$ as
\begin{equation}\label{steinRHS}
d_K(Z,N)\leq \sup_{f_x}\big|\E[f_x'(Z)-Zf_x(Z)]\big|\,,
\end{equation}
where the supremum runs over the class of solutions $f_x$ of \eqref{steineq}. The unique bounded solution of \eqref{steineq} is of the form $$f_x(z)=e^{z^2/2}\int\limits_{-\infty}^z\big({\bf 1}(y\leq x)-P(N\leq x)\big)e^{-y^2/2}\,dy\,,$$ see \cite[Lemma 2.2]{CheGolSha}, and satisfies the estimate $0 < f_x(z)\leq{\sqrt{2\pi}\over 4}$. Moreover $f_x$ is continuous on $\R$, infinitely differentiable on $\R\setminus\{x\}$, but not differentiable at $x$. However, interpreting the derivative of $f_x$ at $x$ as $1 - P(N\leq x)+xf(x)$ in view of \eqref{steineq}, we have
\begin{equation}\label{eq:SteinBoundFirstDerivative}
 \big|f_x'(z)\big| \leq 1\qquad{\rm for\ all}\qquad z\in\R
\end{equation}
from Lemma 2.3 in \cite{CheGolSha}. Moreover, the same result ensures that $f_x$ satisfies
\begin{equation}\label{boundsolution}
 \big|(w+u)f_x(w+u)-(w+v)f_x(w+v)\big|\leq \Big(|w|+{\sqrt{2\pi}\over 4}\Big)(|u|+|v|)
\end{equation}
for all $u,v,w\in\R$.

\subsection{Stein's method for multivariate normal approximation.} 

There is also a multivariate version of Stein's method for normal approximation. It starts with the observation that a centred random variable ${\bf Z}$ with values in $\R^d$ for some $d\geq 2$ follows a multivariate normal distribution with covariance matrix $C$ (which is a positive semi-definite $(d\times d)$-matrix) if and only if
$$\E[\langle {\bf Z},\nabla f({\bf Z})\rangle_{\R^d}-\langle C,{\rm Hess}f({\bf Z})\rangle_{H.S.}]=0$$ for all twice differentiable $f:\R^d\to\R$ with $$\E|\langle {\bf Z},\nabla f({\bf Z})\rangle_{\R^d}|+\E|\langle C,{\rm Hess}f({\bf Z})\rangle_{H.S.}|<\infty\,.$$ Here, $\langle \,\cdot\,,\,\cdot\,\rangle_{\R^d}$ is the inner product in $\R^d$, for two matrices $A$ and $B$, $\langle A,B\rangle_{H.S.}={\rm trace}(AB^T)$ is the Hilbert-Schmidt inner product and ${\rm Hess}f(z)$ stands for the Hessian matrix of $f$ at $z$. If $g:\R^d\to\R$, the multivariate Stein equation reads
\begin{equation}\label{steineqMULTI}
\langle z,\nabla f(z)\rangle_{\R^d}-\langle C,{\rm Hess}f(z)\rangle_{H.S.} = g(x) - \E[g({\bf N})] \,,
\end{equation}
where ${\bf N}$ stands for a random variable with a multivariate centred normal distribution having covariance matrix $C$. It is well known that for a given function $g$,
\begin{equation*}
f_g(x)=\int_0^1{1\over 2t}\,\E[g(\sqrt{t}\,{x}+\sqrt{1-t}\,{\bf N})-g({\bf N})]\,dt\,,\qquad x\in\R^d
\end{equation*}
is a solution of \eqref{steineqMULTI}.
To rephrase smoothness properties of $f_g$ we introduce the following notation. If $h:\R^d\to\R$, $k\in\N$ and $i_1,\ldots,i_k\in\{1,\ldots,d\}$, put
$$M_k(h):=\max_{1\leq i_1,\ldots,i_k\leq d}\,\sup_{x\in\R^d}\Big|{\partial^k\over\partial x_{i_1}\ldots\partial x_{i_k}}h(x)\Big|$$ (provided this is well defined). We notice that
\begin{equation*}
{\partial^k\over\partial x_{i_1}\ldots\partial x_{i_k}}f_g(x)=\int_0^1{1\over 2t}\,t^{k/2}\,\E\Big[{\partial^k\over\partial x_{i_1}\ldots\partial x_{i_k}}g(\sqrt{t}\,{x}+\sqrt{1-t}\,{\bf N})\Big]\,dt\,, 
\end{equation*}
whenever $g$ possesses partial derivatives up to order $k$. In particular, this shows that $M_k(g)\leq 1$ implies $M_k(f_g)\leq 2/k$, see \cite[Lemma 2.6]{CheGolSha}.

To compare (the distributions of) the $\R^d$-valued random variables ${\bf Z}$ and ${\bf N}$, and inspired by \eqref{steineqMULTI}, we use the $d_4$-distance
\begin{align}\label{eq:d4Definition}
d_4({\bf Z},{\bf N}):=\sup_{g}\big|\E g({\bf Z})-\E g({\bf N})\big|\,,
\end{align}
where the supremum runs over all $g:\R^d\to\R$ having continuous partial derivatives up to order $4$ and satisfy $M_i(g)\leq 1$ for $i\in\{1,\ldots,4\}$.



Note that convergence in the $d_4$-distance implies convergence in distribution of the involved random variables.

\section{A one-dimensional Berry-Esseen bound}\label{sec:1dimbound}


We now present the first main result of this work, a Berry-Esseen bound for Radema\-cher functionals in terms of discrete Malliavin operators. From a structural point of view, this bound is very similar to that obtained in Theorem 3.1 of \cite{EicTha}, which is not surprising as we also follow the basic idea of that paper. However, the proof and the interpretation of the involved Malliavin operators are different, because of the special structure of a Rademacher functional. This point will further be discussed in Remark \ref{NewArguments} below.

\begin{theorem}
\label{mainthm}
Let $F\in\text{dom}(D)$ with $\E[F]=0$ and let $N$ be a standard Gaussian random variable. Then
\begin{equation}\label{abstractbound}
\begin{split}
d_K(F,N)&\leq\E[|1-\langle DF,-DL^{-1}F\rangle_{\ell^2(\N)}|]+\frac{\sqrt{2\pi}}{4}\E[\langle(DF)^2,|DL^{-1}F|\rangle_{\ell^2(\N)}]\notag\\
&\quad +\E[\langle(DF)^2,|F\cdot DL^{-1}F|\rangle_{\ell^2(\N)}]+2\,\sup_{x\in\R}\E[\langle (DF)D\1_{\{F>x\}},|DL^{-1}F|\rangle_{\ell^2(\N)}]\,.
\end{split}
\end{equation}
\end{theorem}

\begin{remark}
Our result should be compared with the estimate from \cite{ReiPec}. For a Rademacher functional $F\in\text{dom}(D)$ with $\E[F]=0$, define $B_1(F)$ and $B_2(F)$ by
\begin{align*}
B_1(F) &:= \E[|1-\langle DF,-DL^{-1}F\rangle_{\ell^2(\N)}|]\,,\\
B_2(F) &:= {20\over 3}\E[\langle |DF|^3,|DL^{-1}F|\rangle_{\ell^2(\N)}]\,.
\end{align*}
Then Theorem 3.1 in \cite{ReiPec} says that 
\begin{equation}\label{eq:d3EstimateNRP}
\big|\E[g(F)]-\E[g(N)]\big|\leq \min\{4\|g'\|_\infty,\|g''\|_\infty\}\,B_1(F)+\|g''\|_\infty\,B_2(F)\,,
\end{equation}
where $N$ is a standard Gaussian random variable and $g:\R\to\R$ is a twice differentiable function with bounded derivatives of order one and two. Moreover, using a standard smoothing argument, this has been extended in Corollary 3.6 ibidem to an estimate for the Wasserstein distance:
\begin{equation}\label{eq:dWEstimate}
\begin{split}
d_W(F,N) &:= \sup_{g\in{\rm Lip}_1}|\E[g(F)]-\E[g(N)]|\\
&\leq \sqrt{2(B_1(F)+B_2(F))(5+\E[|F|])}\,,
\end{split}
\end{equation}
where the supremum runs over all Lipschitz functions $g:\R\to\R$ with Lipschitz constant bounded by $1$.
In view of the well-known relation $d_K(F,N)\leq 2\sqrt{d_W(F,N)}$ between the Kolmogorov and the Wasserstein distance, this leads in general to a suboptimal estimate for $d_K(F,N)$. Even \eqref{eq:dWEstimate} is suboptimal compared to \eqref{eq:d3EstimateNRP}. Our bound provided in Theorem \ref{mainthm} resolves this problem. Moreover, in our applications below we will see that the Kolmogorov distance will be of the same order of magnitude as \eqref{eq:d3EstimateNRP}, improving thereby also \eqref{eq:dWEstimate}. 
\end{remark}

\begin{proof}[Proof of Theorem \ref{mainthm}.]
Let $f_x$ be the solution of the Stein equation \eqref{steineq}. In view of \eqref{steinRHS} we have to bound the quantity
\[\E[f_x'(F)-Ff_x(F)]\]
uniformly in $x$. For the following calculation we suppress the dependence on $x$ of the Stein solution and use the abbreviation $f:=f_x$. Using the relation $\delta D=-L$ from \eqref{eq:L=-dD} and the integration-by-parts formula in Lemma \ref{Integration by parts II} we get
\begin{equation}\label{rightsidestein}
\begin{split}
\E[f'(F)-F f(F)] &=\E[f'(F)-\delta(-DL^{-1} F)f(F)]\\
&=\E[f'(F)-\langle Df(F),-DL^{-1}F\rangle_{\ell^2(\N)}]\,.
\end{split}
\end{equation}
We now rewrite the scalar product on the right hand side of \eqref{rightsidestein}. To this end, we find another representation of the gradient $Df(F)$, using the fundamental theorem of calculus. Note that this only makes use of the first derivative of $f$, which is in contrast to the approach taken in \cite{ReiPec}, where an approximate chain rule for $Df(F)$ is used and higher derivatives of $f$ are involved. We get for $k\in\N$,
\begin{align*}
D_kf(F)&=\frac{1}{2}\left((f(F))_k^+-(f(F))_k^-\right)=\frac{1}{2}\left((f(F_k^+))-(f(F_k^-))\right)\\
&=\frac{1}{2}\int_{F_k^--F}^{F_k^+-F}f'(F+t)~dt\\
&=\frac{1}{2}\Big(\int_{F_k^--F}^{F_k^+-F}\big(f'(F+t)-f'(F)\big)~dt+\int_{F_k^--F}^{F_k^+-F}f'(F)~dt\Big)\\
&=\frac{1}{2}\Big(\int_{F_k^--F}^{F_k^+-F}\big(f'(F+t)-f'(F)\big)~dt+f'(F)\cdot(F_k^+-F_k^-)\Big)\,.
\end{align*}
Combining this with \eqref{rightsidestein} we get
\begin{equation}\label{steinNachHDI}
\begin{split}
\E[f'(F)-Ff(F)]&=\E[f'(F)-\langle f'(F)DF,-DL^{-1}F\rangle_{\ell^2(\N)}]\\
&\quad -\E\Big[\Big\langle\frac{1}{2}\int_{F_{(\cdot)}^--F}^{F_{(\cdot)}^+-F}\big(f'(F+t)-f'(F)\big)~dt,-DL^{-1}F\Big\rangle_{\ell^2(\N)}\Big].
\end{split}
\end{equation}
Since $f$ is a solution of the Stein equation, we have for all $t\in\R$,
\[f'(F+t)=(F+t)f(F+t)+\1_{\{F+t\leq x\}}-P(N\leq x).\]
Thus,
\begin{equation}\label{I1+I2}
\begin{split}
\int_{F_k^--F}^{F_k^+-F}\big(f'(F+t)-f'(F)\big)~dt&=\underbrace{\int_{F_k^--F}^{F_k^+-F}\big((F+t)f(F+t)-F f(F)\big)~dt}_{=:I_1(k)}\\
&\qquad +\underbrace{\int_{F_k^--F}^{F_k^+-F}\big(\1_{\{F+t\leq x\}}-\1_{\{F\leq x\}}\big)~dt}_{=:I_2(k)}.
\end{split}
\end{equation}
One can now make use of the bound \eqref{boundsolution}. Using the identities
\begin{equation}\label{difference}
F_k^\pm-F=\pm2D_kF\1_{\{X_k=\mp 1\}}\,,
\end{equation}
 we get
\begin{align}
|I_1(k)|&=\Big|\int_{F_k^--F}^{F_k^+-F}\big((F+t)f(F+t)-F f(F)\big)~dt\Big|\notag\\
&\leq\int\limits_{\min{\{F_k^--F,F_k^+-F\}}}^{\max{\{F_k^--F,F_k^+-F\}}}(|F|+\frac{\sqrt{2\pi}}{4})|t|~dt\notag\\
&=\int\limits_{2\min{\{D_kF\1_{\{X_k=-1\}},-D_kF\1_{\{X_k=1\}}\}}}^{2\max{\{D_kF\1_{\{X_k=-1\}},-D_kF\1_{\{X_k=1\}}\}}}(|F|+\frac{\sqrt{2\pi}}{4})|t|~dt\notag\\
&=\frac{(2D_kF)^2}{2}\,\Big(|F|+\frac{\sqrt{2\pi}}{4}\Big)\,.\label{I1}
\end{align}
For an evaluation of the term $I_2(k)$ we first define
\begin{align*}
I_{+/+}(k)&=\1_{\{X_k=1,D_kF\geq 0\}}\cdot I_2(k)\,,\quad
I_{+/-}(k)=\1_{\{X_k=1,D_kF< 0\}}\cdot I_2(k),\\
I_{-/+}(k)&=\1_{\{X_k=-1,D_kF\geq 0\}}\cdot I_2(k)\,,\ \,
I_{-/-}(k)=\1_{\{X_k=-1,D_kF< 0\}}\cdot I_2(k).
\end{align*}
Using \eqref{difference} we compute $I_{+/+}$ as follows:
\begin{align*}
\left|I_{+/+}\right|&=\Big|\1_{\{X_k=1,D_kF\geq 0\}}\cdot\int_{F_k^--F}^{F_k^+-F}\big(\1_{\{F+t\leq x\}}-\1_{\{F\leq x\}}\big)~dt\Big|\\
&=\1_{\{X_k=1,D_kF\geq 0\}}\cdot\Big|\int_{-2D_kF}^{0}\big(\!\!\!\!\!\underbrace{\1_{\{F+t\leq x\}}}_{\leq \1_{\{F-2D_kF\leq x\}}}-\1_{\{F\leq x\}}\big)~dt\Big|\\
&\leq\1_{\{X_k=1,D_kF\geq 0\}}\cdot \left|\,2\,D_kF\cdot(\1_{\{F-2D_kF\leq x\}}-\1_{\{F\leq x\}})\right|\\
&=\1_{\{X_k=1,D_kF\geq 0\}}\cdot \left|\,2\,D_kF\cdot(\1_{\{F_k^-\leq x\}}-\1_{\{F_k^+\leq x\}})\right|\\
&=\1_{\{X_k=1,D_kF\geq 0\}}\cdot \left|\,2\,D_kF\cdot(\1_{\{F_k^+> x\}}-\1_{\{F_k^-> x\}})\right|\\
&=\1_{\{X_k=1,D_kF\geq 0\}}\cdot \,4\,D_kF\cdot D_k\1_{\{F>x\}}\,.
\end{align*}
Similarly, we get the bounds
\begin{align*}
\left|I_{+/-}(k)\right|&\leq\1_{\{X_k=1,D_kF< 0\}}\cdot \,4\,D_kF\cdot D_k\1_{\{F>x\}}\,,\\
\left|I_{-/+}(k)\right|&\leq\1_{\{X_k=-1,D_kF\geq 0\}}\cdot\, 4\,D_kF\cdot D_k\1_{\{F>x\}}\,,\\
\left|I_{-/-}(k)\right|&\leq\1_{\{X_k=-1,D_kF< 0\}}\cdot \,4\,D_kF\cdot D_k\1_{\{F>x\}}\,.
\end{align*}
Since
\[I_2(k)=I_{+/+}(k)+I_{+/-}(k)+I_{-/+}(k)+I_{-/+}(k)\,,\]
we arrive at 
\begin{equation}
\label{I2}
|I_2(k)|\leq 4\,D_kF\cdot D_k\1_{\{F>x\}}\,.
\end{equation}
Using our bounds $\eqref{I1}$ and $\eqref{I2}$ for $I_1:=(I_1(k))_{k\in\N}$ and $I_2:=(I_2(k))_{k\in\N}$ and the identities $\eqref{steinNachHDI}$ and $\eqref{I1+I2}$ we conclude that
\begin{align*}
\left|\E[f'(F)-F f(F)]\right|&\leq \E\left|1-\langle DF,-DL^{-1}F\rangle_{\ell^2(\N)}\right|+\frac{1}{2}\E[\langle|I_1|+|I_2|,|DL^{-1}F|\rangle_{\ell^2(\N)}]\\
&\leq\E\left|1-\langle DF,-DL^{-1}F\rangle_{\ell^2(\N)}\right|+{\sqrt{2\pi}\over 4}\,\E[\langle (DF)^2,|DL^{-1}F|\rangle_{\ell^2(\N)}]\\
&\quad +\E[\langle(DF)^2,|F\cdot DL^{-1}F|\rangle_{\ell^2(\N)}]+\E[\langle(2DF)\1_{\{F>x\}},|DL^{-1}F|\rangle_{\ell^2(\N)}].
\end{align*}
The proof is completed by taking the supremum over all $x\in\R$.
\end{proof}


Let us finally introduce a simplification of the terms arising in Theorem  \ref{mainthm}, which will be used below. 
\begin{corollary}
\label{MainCorollary}
Let $F \in dom(D)$. Then
\begin{enumerate}[(i)]
\item $\E \big[ \big| 1 - \langle DF, -DL^{-1}F \rangle_{\ell^2(\N)} \big| \big] \leq \big( \E \big[ 
	( 1 - \langle DF, -DL^{-1}F \rangle_{\ell^2(\N)} )^2 \big] \big)^\frac{1}{2}$

\item $\E \big[ \langle (DF)^2, |F \cdot DL^{-1}F| \rangle_{\ell^2(\N)} \big] + 
	\frac{\sqrt{2\pi}}{4} \E \left[ \langle (DF)^2, |DL^{-1}F| \rangle_{\ell^2(\N)} \right]\\
	\leq \big( \E \big[ \langle (DF)^2, (DL^{-1}F)^2 \rangle_{\ell^2(\N)} \big] 
	\big)^\frac{1}{2} \big( \E \big[ \|DF\|_{\ell^2(\N)}^4 \big] \big)^\frac{1}{4} \big( 
	\big( \E [ F^4 ] \big)^\frac{1}{4} + 1 \big)$,
\end{enumerate}
provided all occurring expectations on the right hand sides of the above inequalities are well defined.
\end{corollary}

\begin{proof}
An application of the Cauchy-Schwarz inequality yields (i). Using the Cauchy-Schwarz inequality again with respect to $\E[\,\cdot\,]$ and $\langle\,\cdot\,,\,\cdot\,\rangle_{\ell^2(\N)}$ we find
\begin{align}
&\quad\frac{\sqrt{2\pi}}{4}\E[\langle(DF)^2,|DL^{-1}F|\rangle_{\ell^2(\N)}]+\E[\langle(DF)^2,|F\cdot DL^{-1}F|\rangle_{\ell^2(\N)}]\notag\\
&\leq\E[\langle(DF)^2,(1+|F|)|DL^{-1}F|\rangle_{\ell^2(\N)}]\notag\\
&=\E[\langle(DF)|DL^{-1}F|,(DF)(1+|F|)\rangle_{\ell^2(\N)}]\notag\\
&\leq\E\left[\|(DF)(DL^{-1}F)\|_{\ell^2(\N)}\cdot\|(DF)(1+|F|)\|_{\ell^2(\N)}\right]\notag\\
&\leq\left(\E\left[\langle(DF)^2,(DL^{-1}F)^2\rangle_{\ell^2(\N)}\right]\right)^{\frac{1}{2}}\cdot\big(\E\big[(1+|F|)^2\|DF\|_{\ell^2(\N)}^2\big]\big)^{\frac{1}{2}}\label{mink}.
\end{align}
Now, (ii) is obtained by applying both the Cauchy-Schwarz and the Minkowski inequality to the second term in \eqref{mink}:
\begin{align*}
\big(\E\big[(1+|F|)^2\|DF\|_{\ell^2(\N)}^2\big]\big)^{\frac{1}{2}}
&\leq\E[(1+|F|)^4]^{\frac{1}{4}}\cdot\E[\|DF\|_{\ell^2(\N)}^4]^{\frac{1}{4}}\\
&=\|1+|F|\|_{L^4(\Omega)}\cdot\E[\|DF\|_{\ell^2(\N)}^4]^{\frac{1}{4}}\\
&\leq (1+\E[F^4]^{\frac{1}{4}})\cdot\E[\|DF\|_{\ell^2(\N)}^4]^{\frac{1}{4}}\,.
\end{align*}
This completes the proof.
\end{proof}

\begin{remark}\label{NewArguments}
Let us briefly discuss the novelties of the proof of Theorem \ref{mainthm} compared with the existing literature (such as \cite{ReiPec}). The common thread underlying the Malliavin-Stein approach (also on the Gaussian \cite{NouPecPTRF09} or Poisson space \cite{PecSolTaqUtz}) is the usage of an integration-by-parts formula. For smooth test functions, this is then combined with a Taylor expansion, which leads to an appropriate chain rule. Here, we could not build on the existing so-called approximate chain rule from \cite{ReiPec} and instead followed the idea of \cite{EicTha} by expressing the remainder term in integral form. Handling this term required new estimates, since the Malliavin operator $D$ has a different representation and follows different computation rules in case of Rademacher sequences. 
\end{remark}

\section{Explicit bounds for discrete multiple stochastic integrals}\label{sec:multipleintegrals}

\subsection{The first chaos}

In the present section we apply our abstract bound from Theorem \ref{mainthm} dealing with the Kolmogorov distance between the distribution of a general Rademacher functional $F\in{\rm dom}(D)$ and the standard normal distribution in the case that $F$ belongs to the first Rademacher chaos. This way, we establish a connection to small ball probabilities. So, let $F=J_1(f)$ for some $f\in\ell^2(\N)$, i.e., $F=\sum_{i=1}^\infty f(i)X_i$. Such functionals are known as Rademacher averages in the literature. 

\begin{theorem} 
\label{berryJ1}
Let $F=\sum_{i=1}^\infty a_iX_i$ for $(a_i)_{i\in\N}\in\ell^2(\N)$ such that $\E[F^2]=\sum_{i=1}^\infty a_i^2=1$ and let $N$ be a standard Gaussian random variable. Then
\begin{equation}
d_K(F,N)\leq 2\sum_{i=1}^\infty |a_i|^3+\sup_{x\in\R}\,\sum_{k=1}^\infty a_k^2\cdot P\Big(x-|a_k|<\sum_{\stackrel{i=1}{i\neq k}}^\infty a_iX_i\leq x+|a_k|\Big).\label{smallball}
\end{equation}
\end{theorem}
\begin{proof}
We first introduce abbreviations for the four terms appearing on the right hand side of the bound in Theorem \ref{mainthm}:
\begin{align}
A_1(F)&:=\E[|1-\langle DF,-DL^{-1}F\rangle_{\ell^2{\N}}|]\,,\label{eq:A1original}\\
A_2(F)&:=\frac{\sqrt{2\pi}}{4}\E[\langle(DF)^2,|DL^{-1}F|\rangle_{\ell^2(\N)}]\,,\label{eq:A2original}\\
A_3(F)&:=\E[\langle(DF)^2,|F\cdot DL^{-1}F|\rangle_{\ell^2(\N)}]\,,\label{eq:A3original}\\
A_4(F)&:=2\sup_{x\in\R}\E[\langle (DF)D\1_{\{F>x\}},|DL^{-1}F|\rangle_{\ell^2(\N)}]\,.\label{eq:A4original}
\end{align}
In our case we have that $D_kF=-D_kL^{-1}F=a_k$ for all $k\in\N$ and thus get the following bounds for $A_1(F), A_2(F)$ and $A_3(F)$:
\begin{equation}\label{A1}
\begin{split}
A_1(F)&=\Big|1-\sum_{i=1}^\infty a_i^2\Big|=0\,,\qquad A_2(F)\leq \sum_{i=1}^\infty |a_i|^3\,,\\
A_3(F)&=\E\Big[\sum_{i=1}^\infty |a_i|^3\cdot\Big|\sum_{i=1}^\infty a_i X_i\Big|\Big]\leq\Big(\sum_{i=1}^\infty |a_i|^3\Big)\cdot\sqrt{\E\Big[\Big(\sum_{i=1}^\infty a_iX_i\Big)^2\Big]}=\sum_{i=1}^\infty |a_i|^3\,.
\end{split}
\end{equation}
For the term $A_4(F)$ we first observe that
\begin{align*}
D_k\1\Big(\sum_{i=1}^\infty a_iX_i>x\Big)&=\frac{1}{2}\Big(\1\Big(\sum_{\stackrel{i=1}{i\neq k}}^\infty a_iX_i>x-a_k\Big)-\1\Big(\sum_{\stackrel{i=1}{i\neq k}}^\infty a_iX_i>x+a_k\Big)\Big)\,.
\end{align*}
Thus,
\begin{align}
 A_4(F)
&=2\sup_{x\in\R}\,\E\Big[\sum_{k=1}^\infty D_kF\cdot D_k\1(F>x)|D_kL^{-1}F|\Big]\notag\\
&=\sup_{x\in\R}\,\E\Big[\sum_{k=1}^\infty a_k\cdot|a_k|\Big(\1\Big(\sum_{\stackrel{i=1}{i\neq k}}^\infty a_iX_i>x-a_k\Big)-\1\Big(\sum_{\stackrel{i=1}{i\neq k}}^\infty a_iX_i>x+a_k\Big)\Big)\Big]\notag\\
&=\sup_{x\in\R}\,\E\Big[\sum_{k=1}^\infty a_k^2 \1\Big(x-|a_k|<\sum_{\stackrel{i=1}{i\neq k}}^\infty a_iX_i\leq x+|a_k|\Big)\Big]\notag\\
&=\sup_{x\in\R}\,\sum_{k=1}^\infty a_k^2\cdot P\Big(x-|a_k|<\sum_{\stackrel{i=1}{i\neq k}}^\infty a_iX_i\leq x+|a_k|\Big)\label{A4}.
\end{align}
Putting together \eqref{A1} and \eqref{A4} yields the assertion.
\end{proof}
\begin{remark}
It is interesting to see that the bound \eqref{smallball} in Theorem \ref{berryJ1} involves quantities which are known as small ball probabilities in the literature. More precisely, if $\xi_1,\dots,\xi_n$ are i.i.d.\ real-valued random variables and $a_1,\dots,a_n$ real numbers, then a quantity of the type
\[\sup_{x\in\R}P\Big(\Big|\sum_{i=1}^n a_i\xi_i-x\Big|\leq\varepsilon\Big)\,,\qquad\varepsilon>0\,,\] 
is what is usually called a small ball probability and can be considered as a kind of measure of anti-concentration for the partial sum $\sum_{i=1}^n a_i\xi_i$. The authors in \cite{RudVer} (see also \cite{LitPajRud}) found a bound for these small ball probabilities using the classical Berry-Esseen theorem for i.i.d.\ random variables with finite third moment. In our set-up, i.e., if $X_1,\dots,X_n$ is a sequence of independent Rademacher random variables and if $a_1,\dots,a_n\in\R$, Corollary 2.9 in \cite{RudVer} says that
\begin{equation}
\sup_{x\in\R}P\Big(\Big|\sum_{i=1}^n a_iX_i-x\Big|\leq\varepsilon\Big)\leq\sqrt{\frac{2}{\pi}}\,\varepsilon+C\sum_{i=1}^n|a_i|^3\,,\qquad\varepsilon>0\,,\label{smbound}
\end{equation}
where $0<C<\infty$ is an absolute constant. Using \eqref{smbound} one can see that Theorem \ref{berryJ1} reproduces the correct order for the Berry-Esseen bound for the normal approximation of a finite sum $F= \sum_{i=1}^n a_iX_i$, which is $O(\sum_{i=1}^n|a_i|^3)$. Our Theorem \ref{berryJ1} can be interpreted as an inverse of \eqref{smbound}, as it provides a Berry-Esseen bound for $F$ in terms of small ball probabilities. Also note that our Theorem  \ref{berryJ1} goes beyond this set-up, since it allows $F$ to depend on an infinite sequence of independent Rademacher variables.
\end{remark}

\subsection{The case $q\geq 2$}

Our main result in this section is an estimate for the Kolmogorov distance of a discrete multiple stochastic integral of arbitrary order $q\geq 2$ and a standard Gaussian random variable in terms of contraction norms. We present two estimates, which will separately be used below. We emphasize that they are the discrete analogues of similar results for Gaussian or Poisson multiple stochastic integrals, see \cite{EicTha,NouPecPTRF09}.

\begin{theorem} 
\label{BerryEsseenMultipleIntegrals}
Let $F = J_q(f)$ for a fixed integer $q \geq 2$ and a symmetric kernel $f \in \ell^2_0(\N)^{\circ q}$. Assume that $\|(f\star_r^r f)\1_{\Delta_{2(q-r)}}\|_{\ell^2(\N)^{\otimes 2(q-r)}} < 1$, for all $r=1, \dotsc, q-1$. Furthermore, let $N$ be a standard Gaussian random variable. Then
\begin{align*}
d_K(F, N) &\leq C_1\max\big\{|1-q!\lnorm{2}{q}{f}^2|,\max_{r=1,\ldots,q-1}\{\|(f\star_r^r f)\1_{\Delta_{2(q-r)}}\|_{\ell^2(\N)^{\otimes 2(q-r)}}\},\\
&\qquad\qquad\qquad\qquad\max_{r=1,\ldots,q}\{\|f\star_r^{r-1} f\|_{\ell^2(\N)^{\otimes(2(q-r)+1)}}\}\big\}\\
&\leq C_2\max\big\{|1-q!\lnorm{2}{q}{f}^2|,\max_{r=1,\ldots,q-1}\{\|f\star_r^r f\|_{\ell^2(\N)^{\otimes 2(q-r)}}\}\big\}
\end{align*}
with universal constants $0<C_1,C_2<\infty$ only depending on $q$.
\end{theorem}

\begin{remark}
Theorem \ref{BerryEsseenMultipleIntegrals} only proves useful for applications to sequences $F_n = J_q(f_n)$ with kernels $f_n$ for which at least one of the above bounds vanishes. In such applications, the assumption $\|(f_n \star_r^r f_n)\1_{\Delta_{2(q-r)}}\|_{\ell^2(\N)^{\otimes 2(q-r)}} < 1$, for all $r=1, \dotsc, q-1$, is naturally fullfilled, whenever $n$ is large enough. 
\end{remark}

We prepare the proof of Theorem \ref{BerryEsseenMultipleIntegrals} with the following lemma, which corresponds to Theorem 4.1 in \cite{ReiPec} combined with \eqref{eq:Relation01} and \eqref{Contraction inequality}.
\begin{lemma}\label{A_1 & A_2}
Let $F = J_q(f)$ for a fixed integer $q \geq 2$ and a symmetric kernel $f \in \ell^2_0(\N)^{\circ q}$. Then
\begin{align}
&\E \Big[ \Big( 1 - \frac{1}{q}\| DF \|_{\ell^2(\N)}^2 \Big)^2 \Big] \notag\\
&= \big( 1 - q! \lnorm{2}{q}{f}^2 \big)^2 \notag\\
&\phantom{{}={}} + q^2 \sum_{r=1}^{q-1} \Big( (r-1)! \binom{q-1}{r-1}^2 \Big)^2 (2(q-r))! \,\lnormb{2}{2(q-r)}{( \widetilde{f \star_r^r f} )  \1_{\Delta_{2(q-r)}}}^2\label{eq:HerzAbschaetzung}\\
&\leq \big( 1 - q! \lnorm{2}{q}{f}^2 \big)^2 \notag\\
&\phantom{{}={}}+ q^2 \sum_{r=1}^{q-1} \Big( (r-1)! \binom{q-1}{r-1}^2 \Big)^2 (2(q-r))!\, \lnorm{2}{2(q-r)}{f \star_r^r f}^2\notag
\end{align}
and
\begin{align}
\E \left[ \| DF \|_{\ell^4(\N)}^4 \right]&= \sum_{k=1}^\infty q^4 \sum_{r=1}^{q} \Big( (r-1)! \binom{q-1}{r-1}^2 \Big)^2 (2(q-r))!\notag\\
&\phantom{{}= \sum_{k=1}^\infty q^4 \sum_{r=1}^{q}{}} \times \lnormb{2}{2(q-r)}{( \widetilde{f(\,\cdot\,,k) \star_{r-1}^{r-1} f(\,\cdot\,,k)} ) \1_{\Delta_{2(q-r)}}}^2\notag\\
&\leq q^4 \sum_{r=1}^{q} \Big( (r-1)! \binom{q-1}{r-1}^2 \Big)^2 (2(q-r))! \lnorm{2}{2(q-r)+1}{f \star_r^{r-1} f}^2\notag\\
&\leq C\max\limits_{r=1,\dots,q-1}\{\lnorm{2}{2(q-r)}{f \star_r^{r} f}^2\},\label{eq:4NormHoch4}
\end{align}
for a constant $0<C<\infty$ depending only on $q$.
\end{lemma}
\begin{proof}[Proof of Theorem \ref{BerryEsseenMultipleIntegrals}]
Let us first assume that the support of $f$ satisfies $${\rm supp}(f) \subseteq \lbrace 1, \dotsc, n \rbrace^q$$ for some $n \in \N$. According to Theorem \ref{mainthm} and Corollary \ref{MainCorollary}, we have to bound the following quantities:
\begin{align}
A_1(F) &:= \big( \E \big[ \left( 1 - \langle DF, -DL^{-1}F \rangle_{\ell^2(\N)} \right)^2 \big] \big)^\frac{1}{2}\,,\label{eq:A1} \\
A_2(F) &:= \big( \E \big[ \langle \left( DF \right)^2, ( DL^{-1}F )^2 \rangle_{\ell^2(\N)} \big] \big)^\frac{1}{2}\,, \label{eq:A2}\\
A_3(F) &:= \big( \E \big[ \| DF \|_{\ell^2(\N)}^4 \big] \big)^\frac{1}{4} \big( \big( \E [ F^4 ] \big)^\frac{1}{4} + 1 \big)\,, \label{eq:A3}\\
A_4(F) &:= 2 \,\sup_{x \in \R} \E \big[ \langle ( DF ) ( D\1_{\lbrace F > x \rbrace} ), | DL^{-1}F | \rangle_{\ell^2(\N)} \big]\,.\label{eq:A4}
\end{align}
First observe that, since $D_kF = qJ_{q-1}(f(\cdot, k))$ and $-D_kL^{-1}F = J_{q-1}(f(\cdot, k)) = \frac{1}{q}D_kF$, we have
\begin{align*}
A_1(F) &= \big( \E \big[ \big( 1 - \frac{1}{q}\| DF \|_{\ell^2(\N)}^2 \big)^2 \big] \big)^\frac{1}{2}\quad\text{and}\quad
A_2(F) = \big( \E \big[ \frac{1}{q^2}\| DF \|_{\ell^4(\N)}^4 \big] \big)^\frac{1}{2}\,.
\end{align*}
Thus, using \eqref{eq:Relation01}, \eqref{Contraction inequality} and Lemma \ref{A_1 & A_2}, we can estimate $A_1(F)$ and $A_2(F)$ as follows:
\begin{align}
A_1(F) \notag &\leq  \big| 1 - q! \lnorm{2}{q}{f}^2 \big| \notag\\
&\qquad+ q\Big( \sum_{r=1}^{q-1} \Big( (r-1)! \binom{q-1}{r-1}^2 \Big)^2 (2(q-r))! \lnormb{2}{2(q-r)}{(f \star_r^r f)\1_{\Delta_{2(q-r)}}}^2 \Big)^\frac{1}{2}\label{XA1Bound}\\
&\leq \big| 1 - q! \lnorm{2}{q}{f}^2 \big| \notag\\
&\qquad+ q\Big( \sum_{r=1}^{q-1} \Big( (r-1)! \binom{q-1}{r-1}^2 \Big)^2 (2(q-r))! \lnormb{2}{2(q-r)}{f \star_r^r f}^2 \Big)^\frac{1}{2} \label{A_1 bound}
\end{align}
and
\begin{align}
A_2(F) 
&\leq q\Big( \sum_{r=1}^{q} \Big( (r-1)! \binom{q-1}{r-1}^2 \Big)^2 (2(q-r))! \lnorm{2}{2(q-r)+1}{f \star_r^{r-1} f}^2 \Big)^\frac{1}{2} \label{XA2Bound}\\
&= q\Big( (2(q-1))! \lnormb{2}{2(q-1)+1}{f \star_1^0 f}
\vphantom{\sum_{r=2}^{q} \Big( (r-1)! \binom{q-1}{r-1}^2 \Big)^2} \notag\\
&\qquad+ \sum_{r=2}^{q} \Big( (r-1)! \binom{q-1}{r-1}^2 \Big)^2 (2(q-r))! \lnorm{2}{2(q-r)+1}{f \star_r^{r-1} f}^2 \Big)^\frac{1}{2} \notag\\
&\leq q\Big( (2(q-1))! \lnormb{2}{2}{f \star_{q-1}^{q-1} f}^2\notag\\
&\qquad+ 
\sum_{r=2}^{q} \Big( (r-1)! \binom{q-1}{r-1}^2 \Big)^2 (2(q-r))! \lnorm{2}{2(q-r+1)}{f \star_{r-1}^{r-1} f}^2 \Big)^\frac{1}{2} \notag\\
&= q\Big( (2(q-1))! \lnormb{2}{2}{f \star_{q-1}^{q-1} f}^2\notag\\
&\qquad+ 
 \sum_{r=1}^{q-1} \Big( r! \binom{q-1}{r}^2 \Big)^2 (2(q-r-1))! \lnorm{2}{2(q-r)}{f \star_r^r f}^2 \Big)^\frac{1}{2}\,.\label{A2 bound}
\end{align}

Considering $A_3(F)$, we use the multiplication formula \eqref{Multiplication formula} to see that
\begin{align}
(D_kF)^2 &= q^2 \sum_{r=0}^{q-1} r! \binom{q-1}{r}^2 J_{2(q-r-1)}\big( ( \widetilde{f(\cdot,k) \star_r^r f(\cdot,k)} ) \1_{\Delta_{2(q-r-1)}} \big) \notag\\
&= q^2 \sum_{r=1}^{q} (r-1)! \binom{q-1}{r-1}^2 J_{2(q-r)}\big( ( \widetilde{f(\cdot,k) \star_{r-1}^{r-1} f(\cdot,k)} ) \1_{\Delta_{2(q-r)}} \big). \notag
\end{align}
Since $f$ has finite support, we thus find that
\begin{align}
\| DF \|_{\ell^2(\N)}^2 &= \sum_{k=1}^\infty q^2 \sum_{r=1}^{q} (r-1)! \binom{q-1}{r-1}^2 J_{2(q-r)}\big( ( \widetilde{f(\cdot,k) \star_{r-1}^{r-1} f(\cdot,k)} ) \1_{\Delta_{2(q-r)}} \big) \notag\\
&= q^2 \sum_{r=1}^{q} (r-1)! \binom{q-1}{r-1}^2 J_{2(q-r)}\big( ( \widetilde{f \star_r^r f} ) \1_{\Delta_{2(q-r)}} \big) \notag\\
&= q^2 \sum_{r=1}^{q-1} (r-1)! \binom{q-1}{r-1}^2 J_{2(q-r)} \big( ( \widetilde{f \star_r^r f} ) \1_{\Delta_{2(q-r)}} \big) + q \cdot q! \lnorm{2}{q}{f}^2\,. \notag
\end{align}
Using the isometry relation \eqref{Isometric relation} we deduce the bound
%
\begin{align}
&\big( \E \big[ \| DF \|_{\ell^2(\N)}^4 \big] \big)^\frac{1}{4} \leq  \big(q \cdot q! \lnorm{2}{q}{f}^2 \big)^\frac{1}{2}\notag\\
&\qquad +q \Big( \sum_{r=1}^{q-1} \Big( (r-1)! \binom{q-1}{r-1}^2 \Big)^2 (2(q-r))! \lnormb{2}{2(q-r)}{(f \star_r^r f)\1_{\Delta_{2(q-r)}}}^2 \Big)^\frac{1}{4} \label{XA31Bound}\\
&\leq q \Big( \sum_{r=1}^{q-1} \Big( (r-1)! \binom{q-1}{r-1}^2 \Big)^2 (2(q-r))! \lnormb{2}{2(q-r)}{f \star_r^r f}^2 \Big)^\frac{1}{4} + \big(q \cdot q! \lnorm{2}{q}{f}^2 \big)^\frac{1}{2}\label{A31 bound}
\end{align}
for the first factor in $A_3(F)$. For the second factor we use again the multiplication formula \eqref{Multiplication formula} to see that
\begin{align}
F^2 = \sum_{r=0}^{q-1} r! \binom{q}{r}^2 J_{2(q-r)} \big( ( \widetilde{f \star_r^r f} ) \1_{\Delta_{2(q-r)}} \big) + q! \lnorm{2}{q}{f}^2\,. \notag
\end{align}
The isometry relation \eqref{Isometric relation} then shows that $\E[F^4]$ can be expressed as 
\begin{equation*}
\E[F^4]=\sum_{r=0}^{q-1} \Big( r! \binom{q}{r}^2 \Big)^2 (2(q-r))! \lnormb{2}{2(q-r)}{( \widetilde{f \star_r^r f} ) \1_{\Delta_{2(q-r)}}}^2 + \big( q! \lnorm{2}{q}{f}^2 \big)^2\,.
\end{equation*}
Separating the term $r=0$ and applying \eqref{Taqqu} we find that $\E[F^4]$ is bounded from above by
\begin{align}
&3\big( q! \lnorm{2}{q}{f}^2 \big)^2 \notag\\
&\qquad+ \sum_{r=1}^{q-1} \Big( \Big( q!\binom{q}{r} \Big)^2 + \Big( r!\binom{q}{r}^2 \Big)^2 (2(q-r))! \Big) \lnormb{2}{2(q-r)}{(f \star_r^r f)\1_{\Delta_{2(q-r)}}}^2 \label{XA32Bound}\\
&\leq 3\big( q! \lnorm{2}{q}{f}^2 \big)^2 + \sum_{r=1}^{q-1} \Big( q!\binom{q}{r} \Big)^2 \Big( 1 + \binom{2(q-r)}{q-r} \Big) \lnorm{2}{2(q-r)}{f \star_r^r f}^2\,. \label{A32 bound}
\end{align}


Next, we consider the term $A_4(F)$. By virtue of Lemma \ref{Integration by parts II}, we have
\begin{align}
A_4(F) &= 2 \, \sup_{x \in \R} \E \left[ \langle (DF)(D \1_{\lbrace F > x \rbrace}), |DL^{-1}F| \rangle_{\ell^2(\N)} \right] \notag\\
&= \frac{2}{q} \, \sup_{x \in \R} \E \left[ \langle D\1_{\lbrace F > x \rbrace}, (DF)|DF| \rangle_{\ell^2(\N)} \right] \notag\\
&= \frac{2}{q} \, \sup_{x \in \R} \E \left[ \1_{\lbrace F > x \rbrace} \delta((DF)|DF|) \right] \notag\\
&\leq \frac{2}{q} \,\E \left[ | \delta((DF)|DF|) | \right] \notag\\
&\leq \frac{2}{q} \, (\E \left[ (\delta((DF)|DF|))^2 \right])^\frac{1}{2}\,. \label{A_4(F) inequality}
\end{align}
We now apply the isometry property \eqref{Isometry property} of the divergence operator. This leads to
\begin{align}
&\E \big[ (\delta((DF)|DF|))^2 \big] \notag\\
&= \E \big[ \| (DF) |DF| \|_{\ell^2(\N)}^2 \big] + \E \Big[ \sum_{k,l=1}^\infty [D_k(D_lF|D_lF|)]\,[ D_l(D_kF|D_kF|)] \Big] \notag\\
&= \E \big[ \| DF \|_{\ell^4(\N)}^4 \big] + \E \Big[ \sum_{k,l=1}^\infty [D_k(D_lF|D_lF|)]\,[ D_l(D_kF|D_kF|)] \Big] \notag\\
&\leq \E \big[ \| DF \|_{\ell^4(\N)}^4 \big] + \E \Big[ \sum_{k,l=1}^\infty (D_k(D_lF|D_lF|))^2 \Big]\,. \label{delta inequality}
\end{align}
Now, using both the product formula \eqref{Product formula gradient operator} for the gradient operator and \eqref{stroock} we see that
\begin{align}
&\E \left[ (D_k(D_lF|D_lF|))^2 \right] \notag\\
&= \E \left[ (D_k((D_lF)^2))^2 \1_{\lbrace D_lF \geq 0 \rbrace} + (D_k(-(D_lF)^2))^2 \1_{\lbrace D_lF < 0 \rbrace} \right] \notag\\
&= \E \left[ (D_k((D_lF)^2))^2 \right] \notag\\
&= \E \left[ \left( 2(D_lF)(D_kD_lF) - 2X_k(D_kD_lF)^2 \right)^2 \right] \notag\\
&= 4 \E \left[ (D_lF)^2(D_kD_lF)^2 \right] - 8 \E \left[ X_k(D_lF)(D_kD_lF)^3 \right] + 4 \E \left[ (D_kD_lF)^4 \right] \notag\\
&= 4 \E \left[ (D_lF)^2(D_kD_lF)^2 \right] - 8 \E \left[ D_k((D_lF)(D_kD_lF)^3) \right] + 4 \E \left[ (D_kD_lF)^4 \right] \notag\\
&= 4 \E \left[ (D_lF)^2(D_kD_lF)^2 \right] - 4 \E \left[ (D_kD_lF)^4 \right] \notag\\
&\leq 4 \E \left[ (D_lF)^2(D_kD_lF)^2 \right]. \label{teta inequality}
\end{align}
Combining \eqref{delta inequality} and \eqref{teta inequality} we further estimate
\begin{align}
&\E \big[ (\delta((DF)|DF|))^2 \big] \notag\\
&\leq \E \big[ \| DF \|_{\ell^4(\N)}^4 \big] + 4 \E \Big[ \sum_{l=1}^\infty \Big( (D_lF)^2  \sum_{k=1}^\infty (D_kD_lF)^2 \Big) \Big] \notag\\
&\leq \E \big[ \| DF \|_{\ell^4(\N)}^4 \big] + 4 \E \Big[ \| DF \|_{\ell^4(\N)}^2  \Big( \sum_{l=1}^\infty \Big( \sum_{k=1}^\infty (D_kD_lF)^2 \Big)^2 \Big)^\frac{1}{2} \Big] \notag\\
&\leq \E \big[ \| DF \|_{\ell^4(\N)}^4 \big] + 4 \big( \E \big[ \| DF \|_{\ell^4(\N)}^4 \big] \big)^\frac{1}{2}  \Big( \E \Big[ \sum_{l=1}^\infty \Big( \sum_{k=1}^\infty (D_kD_lF)^2 \Big)^2 \Big] \Big)^\frac{1}{2}. \notag
\end{align}

Putting this into \eqref{A_4(F) inequality}, we arrive at
\begin{align}
A_4(F) 
\leq 2 A_2(F) + 4 \Big( \frac{1}{q} A_2(F) \Big)^\frac{1}{2} \, \underbrace{\Big( \E \Big[ \sum_{l=1}^\infty \Big( \sum_{k=1}^\infty (D_kD_lF)^2 \Big)^2 \Big] \Big)^\frac{1}{4}}_{=: A_4^\prime(F)}. \label{A41 bound}
\end{align}

To bound $A_4'(F)$ further we notice that $D_kD_lF = q(q-1)J_{q-2}(f(\,\cdot\,,k,l))$, which implies
\begin{align}
&\sum_{k=1}^\infty (D_kD_lF)^2 \notag\\
&= \sum_{k=1}^\infty q^2(q-1)^2 \sum_{r=0}^{q-2} r! \binom{q-2}{r}^2 J_{2(q-r-2)} \big( ( \widetilde{f(\cdot,k,l) \star_r^r f(\cdot,k,l)} ) \1_{\Delta_{2(q-r-2)}}\big) \notag\\
&= \sum_{k=1}^\infty q^2(q-1)^2 \sum_{r=2}^{q} (r-2)! \binom{q-2}{r-2}^2 J_{2(q-r)} \big( ( \widetilde{f(\cdot,k,l) \star_{r-2}^{r-2} f(\cdot,k,l)} ) \1_{\Delta_{2(q-r)}}\big) \notag\\
&= q^2(q-1)^2 \sum_{r=2}^{q} (r-2)! \binom{q-2}{r-2}^2 J_{2(q-r)} \big( ( \widetilde{f(\cdot,l) \star_{r-1}^{r-1} f(\cdot,l)} ) \1_{\Delta_{2(q-r)}}\big). \notag
\end{align}
Then, by the isometry of discrete multiple stochastic integrals \eqref{Isometric relation} and the contraction inequality \eqref{Contraction inequality}, we see that
\begin{align}
&\sum_{l=1}^\infty \E \Big[ \Big( \sum_{k=1}^\infty (D_kD_lF)^2 \Big)^2 \Big] \notag\\
&= \sum_{l=1}^\infty q^4(q-1)^4 \sum_{r=2}^{q} \Big( (r-2)! \binom{q-2}{r-2}^2 \Big)^2 (2(q-r))! \notag \lnormb{2}{2(q-r)}{( \widetilde{f(\cdot,l) \star_{r-1}^{r-1} f(\cdot,l)} ) \1_{\Delta_{2(q-r)}}}^2\notag\\
&\leq \sum_{l=1}^\infty q^4(q-1)^4 \sum_{r=2}^{q} \Big( (r-2)! \binom{q-2}{r-2}^2 \Big)^2 (2(q-r))! \lnorm{2}{2(q-r)}{f(\cdot,l) \star_{r-1}^{r-1} f(\cdot,l)}^2 \notag\\
&= q^4(q-1)^4 \sum_{r=2}^{q} \Big( (r-2)! \binom{q-2}{r-2}^2 \Big)^2 (2(q-r))! \lnorm{2}{2(q-r)+1}{f \star_{r}^{r-1} f}^2 \label{XA4Bound}\\
&\leq q^4(q-1)^4 \sum_{r=2}^{q} \Big( (r-2)! \binom{q-2}{r-2}^2 \Big)^2 (2(q-r))! \lnorm{2}{2(q-r+1)}{f \star_{r-1}^{r-1} f}^2 \notag\\
&= q^4(q-1)^4 \sum_{r=1}^{q-1} \Big( (r-1)! \binom{q-2}{r-1}^2 \Big)^2 (2(q-r-1))! \lnorm{2}{2(q-r)}{f \star_{r}^{r} f}^2.\notag
\end{align}

Thus,
\begin{equation}\label{A4Bound}
A_4^\prime(F) \leq q(q-1) \Big( \sum_{r=1}^{q-1} \Big( (r-1)! \binom{q-2}{r-1}^2 \Big)^2 (2(q-r-1))! \lnorm{2}{2(q-r)}{f \star_{r}^{r} f}^2 \Big)^\frac{1}{4}.
\end{equation}
Combining \eqref{XA1Bound}, \eqref{XA2Bound}, \eqref{XA31Bound}, \eqref{XA32Bound}, \eqref{A41 bound} and \eqref{XA4Bound} yields the first inequality in Theorem \ref{BerryEsseenMultipleIntegrals}, while \eqref{A_1 bound}, \eqref{A2 bound}, \eqref{A31 bound}, \eqref{A32 bound}, \eqref{A41 bound} and \eqref{A4Bound} give the second bound in Theorem \ref{BerryEsseenMultipleIntegrals} for a discrete multiple stochastic integral $J_q(f)$ whose integrand $f$ satisfies ${\rm supp}(f) \subseteq \lbrace 1, \dotsc, n \rbrace^q$. For the general case we use the following approximation argument (cf.\ \cite{Pri}). Consider for $n\in\N$ the sequence of truncated kernels $f_{n}:=f\1_{\{1,\dots,n\}^q}$. Since the sequence $(J_q(f_n))_{n\geq 1}$ is a martingale with respect to the filtration $(\mathcal{F}_n)_{n\geq 1}$ with $\mathcal{F}_n:=\sigma(X_1,\dots,X_n)$, an application of the martingale convergence theorem yields that
\[\lim_{n\to\infty}J_q(f\1_{\{1,\dots,n\}^q})=\lim_{n\to\infty}\E[J_q(f)|\mathcal{F}_n]=J_q(f)\,.\]
The assertion thus follows by means of \cite[Lemma 2.6]{Pri} and continuity of the gradient operator on a fixed Rademacher chaos.
\end{proof}

\begin{remark}\label{rem:NormalGeneral}
If in Theorem \ref{BerryEsseenMultipleIntegrals} a general centred Gaussian random variable $N_{\sigma^2}$ with variance $\sigma^2>0$ is used, the Berry-Esseen bound has to be replaced by
\[d_K(F, N_{\sigma^2}) \leq C\,{1\over \sigma^2}\max\big\{|\sigma^2-q!\lnorm{2}{q}{f}^2|,\max_{r=1,\ldots,q-1}\{\|f\star_r^r f\|_{\ell^2(\N)^{\otimes 2(q-r)}}\}\big\}\,.\]
This is easily verified by a re-scaling argument.
\end{remark}

\subsection{A necessary condition for double integrals}

Theorem \ref{BerryEsseenMultipleIntegrals} says that $J_q(f_n)$ converges in distribution to a standard Gaussian random variable if
$$\lim_{n\to\infty}\max_{r=1,\ldots,q-1}\{\|f_n\star_r^r f_n\|_{\ell^2(\N)^{\otimes 2(q-r)}}\}=0\,,$$
provided that $q!\|f_n\|\to 1$, as $n\to\infty$.
In view of the results from \cite{EicTha,LacPec1,NouPecPTRF09} implying that vanishing contraction norms yield a central limit theorem for Gaussian or Poisson multiple integrals (in the latter case at least if the functions $f_n$ are non-negative), it is natural to ask whether this is also a necessary condition for a sequence of discrete multiple integrals to satisfy a central limit theorem. Here, we concentrate on the case of double integrals $F_n:=J_2(f_n)$ and recall that it has been claimed in Proposition 4.6 of \cite{ReiPec} that $$\lim_{n\to\infty}\lnorm{2}{2}{f_n \star_1^1 f_n}^2 = 0$$ is a necessary and sufficient condition for weak convergence of the distribution of $F_n$ to the standard normal distribution, if $\E[F_n^2]\to 1$, as $n\to\infty$. The proof of sufficiency of this condition for asymptotic normality of (quite general) quadratic forms goes back to the classical work \cite{Jon} of de Jong. Another paper in this regard is the paper \cite{Cha} of Chatterjee, where in Proposition 3.1 ibidem a bound for the normal approximation of quadratic forms of Rademacher functionals has been obtained by means of Stein's method. As shown in \cite{ReiPec}, these bounds also have a representation in terms of norms of contractions. However, it turns out that the convergence of these norms to zero is not neccessary for asymptotic normality as the following example shows (see Section 1.6 in \cite{NouPecReiInvariance}).

\begin{example}\label{ex:Gegenbsp}\rm 
Consider $F_n = J_2(f_n)$, $n \geq 2$, with
\begin{align*}
f_n(i,j) = \begin{cases}
\frac{1}{2\sqrt{n-1}} \,, &\text{if $\{ i, j \} = \{ 1, k \}$ for some $k=2, \dotsc, n$}\,,\\
0\,, &\text{otherwise}.
\end{cases}
\end{align*}
Then, $$F_n = \sum_{i,j=1}^n f_n(i,j)X_iX_j = \frac{X_1}{\sqrt{n-1}} \sum_{i=2}^n X_i$$ with $\E[F_n^2] = 1$. We notice that the distribution of $F_n$ converges weakly to the standard normal distribution. But
\begin{align}
(f_n \star_1^1 f_n) (i,j) &= \sum_{k=2}^n f_n^2(1,k) \1_{\{ i=j=1 \}}(i,j) + f_n(1,i)f_n(1,j) \1_{\{ i \geq 2, j \geq 2 \}}(i,j) \notag\\
&= \frac{1}{4} \1_{\{ i=j=1 \}}(i,j) + \frac{1}{4(n-1)} \1_{\{ i \geq 2, j \geq 2 \}}(i,j) \notag
\end{align}
and hence
\begin{align}
&\lnorm{2}{2}{f_n \star_1^1 f_n}^2 = \sum_{i,j=1}^n \left( \frac{1}{16} \1_{\{ i=j=1 \}}(i,j) + \frac{1}{16(n-1)^2} \1_{\{ i \geq 2, j \geq 2 \}}(i,j) \right) = \frac{1}{8}. \notag
\end{align}
\end{example}

We now deduce a new necessary condition for a sequence $J_2(f_n)$ of double integrals to converge in distribution to a standard Gaussian random variable. It shows that the validity of such a central limit theorem depends in general on a subtle interplay of contraction norms on and off diagonals. Such a phenomenon is not visible for Gaussian or Poisson multiple integrals and seems to be a special feature of the discrete set-up.

\begin{theorem}\label{prop:Necessary}
Let $F_n = J_2(f_n)$, $f_n \in \ell_0^2(\N)^{\circ 2}$ and assume that $\E [F_n^2 ]= 1$ for all $n\in\N$. A necessary condition for the convergence of the distribution of $F_n$ to the standard normal distribution is that
\begin{align}
2\lnorm{4}{2}{f_n}^4 + 3\left( \lnorm{2}{2}{\left( f_n \star_1^1 f_n \right) \1_{\Delta_2}}^2 - \lnorm{2}{2}{\left( f_n \star_1^1 f_n \right) \1_{\Delta_2^c}}^2 \right) \longrightarrow 0\,, \label{Necessary condition}
\end{align}
as $n \rightarrow \infty$.
\end{theorem}
It is readily checked that our new necessary condition \eqref{Necessary condition} is satisfied for Example \ref{ex:Gegenbsp}.

\medspace

The proof of Theorem \ref{prop:Necessary} is prepared by the following two lemmas. The first one is a hypercontractivity property of Rademacher functionals with a finite chaotic decomposition and the second one provides an expression for the 4th moment of a discrete stochastic double integral.
\begin{lemma}
\label{Hypercontractivity}
Let $F = \E[F] + \sum_{n=1}^d J_n(f_n)$, $f_n\in\ell^2_0(\N)^{\circ n}$, for some $d \in \N$, and suppose that $2 \leq p < q$.  Then there exists a constant $0<C<\infty$ depending only on $d$ such that
\begin{align}
\left( \E \left[ \left| F \right|^q \right] \right)^\frac{1}{q} \leq C \left( \frac{q-1}{p-1} \right)^\frac{d}{2} \left( \E \left[ \left| F \right|^p \right] \right)^\frac{1}{p}\,. 
\end{align}
\end{lemma}
\begin{proof}
For $F$ depending only on finitely many Rademacher variables this is Theorem 3.2.5 in \cite{PenGin}. In the general case, the assertion follows by means of an approximation argument and Corollary 0.2.1 in \cite{KwaWoy}.
\end{proof}

\begin{lemma}
Let $F = J_2(f)$ with $f \in \ell_0^2(\N)^{\circ 2}$ such that $\E[F^2] = 1$. Then,
\begin{align}
\E[F^4] &= 3 + 32\lnorm{4}{2}{f}^4 + 48\left( \lnorm{2}{2}{\left( f \star_1^1 f \right) \1_{\Delta_2}}^2 - \lnorm{2}{2}{\left( f \star_1^1 f \right) \1_{\Delta_2^c}}^2 \right). \label{4th moment}
\end{align}
\end{lemma}

\begin{proof}
From the multiplication formula \eqref{Multiplication formula}, we have
\begin{align}
F^2 = J_4 \big( ( \widetilde{f \star_0^0 f} ) \1_{\Delta_4} \big) + 4J_2\big( ( f \star_1^1 f ) \1_{\Delta_2} \big) + 2\lnorm{2}{2}{f}^2. \notag
\end{align}

Now the isometry of multiple stochastic integrals \eqref{Isometric relation} and relation \eqref{Taqqu} yield
\begin{align}
\E[F^4] &= 24 \, \big\| ( \widetilde{f \star_0^0 f} ) \1_{\Delta_4} \big\|_{\ell^2(\N)^{\otimes 4}}^2 + 32\lnormb{2}{2}{( f \star_1^1 f ) \1_{\Delta_2}}^2 + 4\lnorm{2}{2}{f}^4 \notag\\
&= 24 \, \big\| \widetilde{f \star_0^0 f} \big\|_{\ell^2(\N)^{\otimes 4}}^2 - 24 \, \big\| ( \widetilde{f \star_0^0 f} ) \1_{\Delta_4^c} \big\|_{\ell^2(\N)^{\otimes 4}}^2 + 32\lnormb{2}{2}{\left( f \star_1^1 f \right) \1_{\Delta_2}}^2 \notag\\
&\qquad\qquad+ 4\lnorm{2}{2}{f}^4 \notag\\
&= 12\lnormb{2}{2}{f}^4 + 48\lnormb{2}{2}{( f \star_1^1 f ) \1_{\Delta_2}}^2 + 16\lnormb{2}{2}{( f \star_1^1 f ) \1_{\Delta_2^c}}^2 \notag\\
&\qquad\qquad-24 \, \big\| ( \widetilde{f \star_0^0 f} ) \1_{\Delta_4^c} \big\|_{\ell^2(\N)^{\otimes 4}}^2\,. \label{E[F^4]}
\end{align}
By definition we have
\[\big\| ( \widetilde{f \star_0^0 f} ) \1_{\Delta_4^c} \big\|_{\ell^2(\N)^{\otimes 4}}^2 
= \sum_{i_1,i_2,i_3,i_4} \big( ( \widetilde{f \star_0^0 f} )(i_1,i_2,i_3,i_4) \big)^2 \1_{\Delta_4^c}(i_1,i_2,i_3,i_4),\]
with the indices running over all non-negative integers.
Since $f$ is symmetric and vanishes on diagonals, there are only two types of $4$-tuples $(i_1,i_2,i_3,i_4)$ for which $(\widetilde{f \star_0^0 f})(i_1,i_2,i_3,i_4)$ can be non-zero. They are of the form $(i_1,i_1,i_2,i_3)$ and $(i_1,i_1,i_2,i_2)$ -- up to permutation of the entries. Hence,
\begin{align}
&\sum_{i_1,i_2,i_3,i_4} \big( ( \widetilde{f \star_0^0 f} )(i_1,i_2,i_3,i_4) \big)^2 \1_{\Delta_4^c}(i_1,i_2,i_3,i_4)\notag\\
&= 6\sum_{i_1,i_2,i_3 \atop i_2\neq i_3} \big( ( \widetilde{f \star_0^0 f} )(i_1,i_1,i_2,i_3) \big)^2 + 3\sum_{i_1,i_2} \big( ( \widetilde{f \star_0^0 f} )(i_1,i_1,i_2,i_2) \big)^2 \notag\\
&= 6\sum_{i_1,i_2,i_3} \big( ( \widetilde{f \star_0^0 f} )(i_1,i_1,i_2,i_3) \big)^2 - 3\sum_{i_1,i_2} \big( ( \widetilde{f \star_0^0 f} )(i_1,i_1,i_2,i_2) \big)^2\,.\label{Contr. comb.}
\end{align}

Note that for every $4$-tupel $(i_1,i_2,i_3,i_4)$ there are $2^{3}$ permutations $\sigma$ such that $$f(i_1,i_2)f(i_3,i_4)=f(i_{\sigma(1)},i_{\sigma(2)})f(i_{\sigma(3)},i_{\sigma(4)})\,.$$ Using this together with the symmetry of $f$, we deduce that
\begin{align}
( \widetilde{f \star_0^0 f} )(i_1,i_2,i_3,i_4)
&= \frac{2^3}{4!}\big( f(i_1,i_2)f(i_3,i_4) + f(i_1,i_3)f(i_2,i_4) + f(i_1,i_4)f(i_2,i_3)\big) \label{Contr. symm.}
\end{align}
for all $i_1,i_2,i_3,i_4\in\N$.
Combining \eqref{Contr. comb.} and \eqref{Contr. symm.}, we get
\begin{align}
\big\| ( \widetilde{f \star_0^0 f} ) \1_{\Delta_4^c} \big\|_{\ell^2(\N)^{\otimes 4}}^2 &= \frac{8}{3} \sum_{i_1,i_2,i_3} (f(i_1,i_2) f(i_1,i_3))^2 - \frac{4}{3} \sum_{i_1,i_2} f(i_1,i_2)^4 \notag\\
&= \frac{8}{3} \lnormb{2}{2}{( f \star_1^1 f ) \1_{\Delta_2^c}}^2 - \frac{4}{3} \lnorm{4}{2}{f}^4\,. \label{Contr. equality}
\end{align}

Plugging \eqref{Contr. equality} into \eqref{E[F^4]}, we conclude that
\begin{align*}
\E[F^4] &= 12\lnorm{2}{2}{f}^4 + 32\lnorm{4}{2}{f}^4 \notag\\
&\qquad\qquad+ 48\big( \lnorm{2}{2}{( f \star_1^1 f ) \1_{\Delta_2}}^2 - \lnorm{2}{2}{( f \star_1^1 f ) \1_{\Delta_2^c}}^2 \big)\,.
\end{align*}

Equation \eqref{4th moment} now follows immediately from our assumption that $\E [F^2] = 2\lnorm{2}{2}{f}^2 = 1$.
\end{proof}

\begin{proof}[Proof of Theorem \ref{prop:Necessary}]
Since $F_n$ converges in distribution to a standard Gaussian random variable and since $\sup_{n \in \N} \E \left[ \left| F_n \right|^q \right] < \infty$ for all $q \geq 2$, due to the hypercontractivity property stated in Lemma \ref{Hypercontractivity}, we have $\E[F_n^4] \rightarrow 3$. 
The statement now follows from \eqref{4th moment}.
\end{proof}

\begin{remark}
It is a natural question whether the proof of Theorem \ref{BerryEsseenMultipleIntegrals} can be modified in such a way that it involves differences of the type 
\begin{equation}\label{eq:BegruendungBSP}
\lnormb{2}{2(q-r)}{(f\star_r^rf)\1_{\Delta_{2(q-r)}}}-\lnormb{2}{2(q-r)}{(f\star_r^rf)\1_{\Delta_{2(q-r)}^c}}\,.
\end{equation}
rather than the contraction norms $\|f\star_r^rf\|_{\ell^2(\N)^{\otimes 2(q-r)}}$.
We doubt that this is possible with the usual Malliavin-Stein technique, since already the key term $A_1(F_n)$ at \eqref{eq:A1} contains only the off-diagonal term in \eqref{eq:BegruendungBSP}. For Example \ref{ex:Gegenbsp} with $F_n=J_2(f_n)$, $A_1(F_n)$ reduces to $$A_1(F_n)^2=\E[(1-{1\over 2}\|DF_n\|_{\ell^2(\N)})^2]=8\lnorm{2}{2}{(f_n\star_1^1 f_n)\,\1_{\Delta_2}}^2={1\over 16}{n-2\over n-1}\,.$$ This converges to ${1\over 16}\neq 0$, as $n\to\infty$, and since the other terms $A_2(F_n)$, $A_3(F_n)$ and $A_4(F_n)$ in the Malliavin-Stein bound (recall \eqref{eq:A2}, \eqref{eq:A3} and \eqref{eq:A4}) are non-negative, $d_K(F_n,N)$ does not converge to zero. Since the variance comparison in terms of $A_1(F_n)$ lies at the heart of the Malliavin-Stein method, new ideas are necessary to overcome this difficulty.
\end{remark}

\subsection{Sums of single and double integrals}

For one of our applications below we need an estimate for the Kolmogorov distance between a Rademacher functional of the type $$F:=J_1(f^{(1)})+J_2(f^{(2)})\quad\text{with}\quad f^{(1)}\in\ell^2(\Z)\,,f^{(2)}\in\ell^2_0(\Z)^{\circ 2}$$ and a standard Gaussian random variable $N$. In other words, $F$ is the sum of an element of the first and an element of the second Rademacher chaos. From a technical point of view, such functionals are more elaborate compared to elements of a single Rademacher chaos. We take up this point and develop a bound for $d_K(F,N)$.

As indicated in \cite{ReiPec}, the results of discrete Malliavin calculus as outlined above and Stein's method extend to Rademacher random variables indexed by $\Z$. In this section and also in our first application presented in Section \ref{sec:applications} below we make use of this extension in order to avoid boundary effects, following thereby \cite{ReiPec}.

\begin{theorem}\label{thm:SumSingle+Double}
Let $F=J_1(f^{(1)})+J_2(f^{(2)})$ with $f^{(1)}\in\ell^2(\Z)$ and $f^{(2)}\in\ell^2_0(\Z)^{\circ 2}$, and let $N$ be a standard Gaussian random variable. Suppose that $\E[F^2]=1$, then
\begin{align*}
d_K(F,N) &\leq 3\|f^{(1)}\star_1^1f^{(2)}\|_{\ell^2(\Z)}+2\sqrt{2}\|(f^{(2)}\star_1^1f^{(2)})\1_{\Delta_2}\|_{\ell^2(\Z)^{\otimes 2}}+2\|f^{(1)}\|_{\ell^4(\Z)}^2\\
&\qquad +(4\sqrt{2}+12)\|(f^{(2)}\star_1^1 f^{(2)})\1_{\Delta_2^c}\|_{\ell^2(\Z)}+(2\sqrt{13}+6)\,\bigg(\sum_{k,j\in\Z}f^{(1)}(k)^2\,f^{(2)}(k,j)^2\bigg)^{\frac{1}{2}}\\
&\qquad+2\sum_{k\in\Z}\Big(|f^{(1)}(k)|+2\sum_{j\in\Z}|f^{(2)}(j,k)|\Big)^3\,.
\end{align*}
\end{theorem}
\begin{remark}
We should compare our result to the bound of Proposition 5.1 in \cite{ReiPec}. It has been shown there that for a twice differentiable function $g:\R\to\R$ with bounded derivatives up to order two the estimate
\begin{align*}
\big|\E[g(F)]-\E[g(N)]\big| &\leq C(g)\,\big(3\|f^{(1)}\star_1^1 f^{(2)}\|_{\ell^2(\Z)}+2\sqrt{2}\|(f^{(2)}\star_1^1 f^{(2)})\1_{\Delta_2}\|_{\ell^2(\Z)^{\otimes 2}}\big)\\
&\qquad+{160\over 3}\|g''\|_\infty\sum_{k\in\Z}\Big[f^{(1)}(k)^4+16\Big(\sum_{i\in\Z}|f^{(2)}(i,k)|\Big)^4\Big]
\end{align*}
holds with $C(g):=\min\{4\|g'\|_\infty,\|g''\|_\infty\}$. The fact that our bound is more involved is not surprising since already the abstract bound in Theorem \ref{mainthm} contains more terms compared to \eqref{eq:d3EstimateNRP} from \cite{ReiPec}, because our bound is based on the fundamental theorem of calculus rather than on an approximate chain rule. As already seen in Theorem \ref{mainthm} this also leads to different exponents in our bound compared to the results from \cite{ReiPec}.
\end{remark}
\begin{proof}[Proof of Theorem \ref{thm:SumSingle+Double}.]
In view of our general Berry-Esseen estimate in Theorem \ref{mainthm} we have to bound the quantities $A_1(F)$--$A_4(F)$ given by \eqref{eq:A1original}--\eqref{eq:A4original} with $F=J_1(f^{(1)})+J_2(f^{(2)})$ there. The first term $A_1(F)$ has already been addressed in \cite[Proposition 5.1]{ReiPec}:
\begin{equation}\label{eq:A1Sum}
A_1(F) \leq 3\|f^{(1)}\star_1^1f^{(2)}\|_{\ell^2(\Z)}+2\sqrt{2}\|(f^{(2)}\star_1^1f^{(2)})\1_{\Delta_2}\|_{\ell^2(\Z)^{\otimes 2}}\,.
\end{equation}
Using the estimates $$|D_kF|\leq|f^{(1)}(k)|+2\sum_{j\in\Z}|f^{(2)}(j,k)|\quad\text{and}\quad|D_kL^{-1}F|\leq|f^{(1)}(k)|+2\sum_{j\in\Z}|f^{(2)}(j,k)|\,,$$ valid for all $k\in\N$, we find that (replacing thereby also $\sqrt{2\pi}/4$ by $1$)
\begin{equation}\label{eq:A2Sum}
A_2(F)\leq \sum_{k\in\Z}\Big(|f^{(1)}(k)|+2\sum_{j\in\Z}|f^{(2)}(j,k)|\Big)^3
\end{equation}
and, in addition, using the Cauchy-Schwarz inequality together with our assumption that $\E[F^2]=1$,
\begin{equation}\label{eq:A3Sum}
A_3(F)\leq \sum_{k\in\Z}\Big(|f^{(1)}(k)|+2\sum_{j\in\Z}|f^{(2)}(j,k)|\Big)^3\,.
\end{equation}
It remains to consider the term $A_4(F)$. Following the strategy already used in the proof of Theorem \ref{BerryEsseenMultipleIntegrals} we find that
\begin{align*}
&\E[\langle(DF)D\1_{\{F>x\}},(DF)|DL^{-1}F|\rangle_{\ell^2(\Z)}]  \leq \big(\E[(\delta((DF)|DL^{-1}F|))^2]\big)^{1/2}\\
& \leq \big(\E[\|(DF)|DL^{-1}F|\|_{\ell^2(\Z)}^2]\big)^{\frac 12} + \bigg(\E\Big[\sum_{k,l\in\Z}(D_k((D_lF)|D_lL^{-1}F|))\,(D_l((D_kF)|D_kL^{-1}F|))\Big]\bigg)^{\frac 12}\\
&\leq \big(\E[\|(DF)|DL^{-1}F|\|_{\ell^2(\Z)}^2]\big)^{\frac 12} + \bigg(\E\Big[\sum_{k,l\in\Z}(D_k((D_lF)|D_lL^{-1}F|))^2\Big]\bigg)^{\frac 12}\,,
\end{align*}
where we have used the isometric relation \eqref{Isometric relation} for the divergence operator. Now, let us consider an individual summand
$$\E[(D_k(D_lF)|D_lL^{-1}F|))^2] = \E[(D_k((D_lF)(D_lL^{-1}F)))^2]\,.$$
Using the product formula \eqref{Product formula gradient operator} we see that this equals $$\E\big[\big((D_kD_lF)(D_lL^{-1}F)+(D_lF)(D_kD_lL^{-1}F)-2X_k(D_kD_lF)(D_kD_lL^{-1}F)\big)^2\big]\,.$$
Taking into account that
$$D_kF=f^{(1)}(k)+2\, J_1(f^{(2)}(\,\cdot\,,k))\quad\text{and}\quad D_kL^{-1}F=-f^{(1)}(k)- J_1(f^{(2)}(\,\cdot\,,k))\,,\quad k\in\N\,,$$
we get, after simplifications using the isometry property \eqref{Isometry property} of stochastic integrals,
\begin{align*}
\E[(D_k((D_lF)(D_lL^{-1}F)))^2] &= 9f^{(1)}(l)^2f^{(2)}(k,l)^2+16 f^{(2)}(k,l)^2\,\sum_{j\in\Z} f^{(2)}(j,l)^2 - 16 f^{(2)}(k,l)^4\,\\
&\leq  9f^{(1)}(l)^2f^{(2)}(k,l)^2+16 f^{(2)}(k,l)^2\,\sum_{j\in\Z} f^{(2)}(j,l)^2.
\end{align*}
Next, we notice that $\E[J_1(f^{(2)}(\,\cdot\,,k))^3]=0$ and $$\E[J_1(f^{(2)}(\,\cdot\,k))^4]=2\,\sum_{i,j\in\Z\atop i\neq j}f^{(2)}(i,k)^2f^{(2)}(j,k)^2+\Big(\sum_{j\in\Z}f^{(2)}(j,k)\Big)^2\,,$$ for all $k\in\Z$, as a consequence of the multiplication formula \eqref{Multiplication formula}. Using this together with the representations of $D_kF$ and $D_kL^{-1}F$ from above, we see that
\begin{align*}
\E[(D_kF)^2(D_kL^{-1}F)^2] &= f^{(1)}(k)^4+13f^{(1)}(k)^2\,\sum_{j\in\Z} f^{(2)}(j,k)^2\\
&\qquad+4\Big(2\,\sum_{i,j\in\Z\atop i\neq j} f^{(2)}(i,k)^2f^{(2)}(j,k)^2+\Big(\sum_{j\in\Z}f^{(2)}(j,k)^2\Big)^2\,\Big)\,.
\end{align*}
Hence, we find that $A_4(F)$ is bounded by
\begin{align*}
A_4(F)&\leq 2\,\big(\E[\|(DF)|DL^{-1}F|\|_{\ell^2(\Z)}^2]\big)^{\frac 12} + 2\bigg(\E\Big[\sum_{k,l\in\Z}(D_k((D_lF)|D_lL^{-1}F|))^2\Big]\bigg)^{\frac 12}\\
&\leq 2\bigg(\sum_{k\in\Z} f^{(1)}(k)^4\bigg)^{\frac 12}+2\bigg(13\,\sum_{k,j\in\Z} f^{(1)}(k)^2f^{(2)}(j,k)^2\bigg)^{\frac 12}\\
&\qquad +4\bigg(2\,\sum_{k\in\Z}\sum_{i,j\in\Z\atop i\neq j} f^{(2)}(i,k)^2f^{(2)}(j,k)^2\bigg)^{\frac 12}\\
&\qquad+4\,\bigg(\sum_{k\in\Z}\Big(\sum_{j\in\Z} f^{(2)}(j,k)^2\Big)^2\bigg)^{\frac 12}+6\bigg(\sum_{k,l\in\Z} f^{(1)}(l)^2f^{(2)}(k,l)^2\bigg)^{\frac 12}\\
&\qquad\qquad+8\,\bigg(\sum_{k,j,l\in\Z} f^{(2)}(k,l)^2f^{(2)}(j,k)^2\bigg)^{\frac 12}\\
&\leq 2\bigg(\sum_{k\in\Z} f^{(1)}(k)^4\bigg)^{\frac 12}+(2\sqrt{13}+6)\,\bigg(\sum_{k,l\in\Z} f^{(1)}(k)^2f^{(2)}(l,k)^2\bigg)^{\frac 12}\\
&\qquad (4\sqrt{2}+12)\,\bigg(\sum_{k\in\Z}\Big(\sum_{i\in\Z} f^{(2)}(i,k)^2\Big)^2\bigg)^{\frac 12}\,.\\
\end{align*}
Re-writing (whenever this is possible) the sums in terms of norms of suitable contractions completes the proof.
\end{proof}

\section{Multivariate limit theorems for Rademacher functionals}\label{sec:multivariate}

In the present section we compare a vector of Rademacher functionals with a multivariate Gaussian random variable. Recall that for this purpose we use the $d_4$-distance introduced at \eqref{eq:d4Definition}.

\begin{theorem}\label{thm:MultiGeneral}
Fix an integer $d\geq 2$, let $F_1,\ldots,F_d\in{\rm dom}(D)$ be Rademacher functionals satisfying $\E[F_i]=0$ and $\E[\|DF_i\|^4_{\ell^4(\N)}]<\infty$ for all $i\in\{1,\ldots,d\}$ and define the random vector ${\bf F}:=(F_1,\ldots,F_d)$. Moreover, let $\Sigma=(\sigma_{ij})_{i,j=1}^d$ be a positive semi-definite symmetric $(d\times d)$-matrix and ${\bf N}$ a centred $d$-dimensional Gaussian random vector with covariance matrix $\Sigma$. Then
\begin{equation}\label{eq:d4BoundGeneral}
\begin{split}
d_4({\bf F},{\bf N}) &\leq {d\over 2}\Big(\sum_{i,j=1}^d\E\big[(\sigma_{ij}-\langle DF_j,-DL^{-1}F_i\rangle_{\ell^2(\N)})^2\big]\Big)^{1/2}\\
&\qquad\qquad\qquad\qquad+{5\over 3}\E\big[\big\langle\big(\sum_{j=1}^d|DF_j|\big)^3,\sum_{i=1}^d|DL^{-1}F_i|\big\rangle_{\ell^2(\N)}\big]\,.
\end{split}
\end{equation}
\end{theorem}

The proof of Theorem \ref{thm:MultiGeneral} is given below. A particular case arises if each of the Rademacher functionals has the form of a discrete multiple stochastic integral. Then, Theorem \ref{thm:MultiGeneral} implies the following multivariate analogue of Theorem \ref{BerryEsseenMultipleIntegrals}, whose proof is postponed to the end of this section.

\begin{corollary}\label{cor:MutiIntegrals}
Fix an integer $d\geq 2$ and $q_1,\ldots,q_d\in\N$. Further, let for each $i\in\{1,\ldots,d\}$, $f^{(i)}\in\ell_0^2(\N)^{\circ q_i}$, define the random vector ${\bf F}:=(J_{q_1}(f^{(1)}),\ldots,J_{q_d}(f^{(d)}))$ and denote by ${\bf N}$ a centred Gaussian random vector with covariance matrix $\Sigma=(\sigma_{ij})_{i,j=1}^d$, such that $\sigma_{ij}=0$ if $q_i\neq q_j$. Then,
\begin{align}
d_4({\bf F},{\bf N}) &\leq C_1\max_{i,j=1,\ldots,d}\Big\{\big|\sigma_{ij}-\E[J_{q_i}(f^{(i)})J_{q_j}(f^{(j)})]\big|\,,\notag\\
&\qquad\qquad\max_{r=1,\ldots,\min\{q_i,q_j\}}\big\{\lnormb{2}{q_i+q_j-2r}{f^{(i)}\star_r^r f^{(j)}}\big\}\1_{\{q_i\neq q_j\}}\, ,\notag\\
&\qquad\quad\qquad \max_{r=1,\ldots,q_i-1}\big\{\lnormb{2}{2(q_i-r)}{f^{(i)}\star_r^r f^{(j)}}\big\}\1_{\{q_i= q_j\}}\, ,\notag\\
&\qquad\qquad\qquad \max_{r=1,\ldots,q_i-1}\big\{\lnormb{2}{2(q_i-r)}{f^{(i)}\star_r^r f^{(i)}}^2\big\}\Big\}\label{eq:Cor1}\\
&\leq C_2\!\max_{i,j=1,\ldots,d}\Big\{\big|\sigma_{ij}-\E[J_{q_i}(f^{(i)})J_{q_j}(f^{(j)})]\big|\,,\notag\\
&\qquad\qquad\max_{r=1,\ldots,q_i-1}\big\{\lnormb{2}{2(q_i-r)}{f^{(i)}\star_r^r f^{(i)}},\lnormb{2}{2(q_i-r)}{f^{(i)}\star_r^r f^{(i)}}^2,\notag\\
&\qquad\qquad\qquad\qquad\qquad\lnormb{2}{q_j}{f^{(j)}}\lnormb{2}{2(q_i-r)}{f^{(i)}\star_r^rf^{(i)}}^{1/2}\big\}\Big\}\label{eq:Cor2}
\end{align}
with universal constants $0<C_1,C_2<\infty$ only depending on $d$ and on $q_1,\ldots,q_d$.
\end{corollary}

The second estimate in Corollary \ref{cor:MutiIntegrals} especialy implies that a random vector $$(J_{q_1}(f_n^{(1)}),\ldots,J_{q_d}(f_n^{(d)}))\,,\qquad f_n^{(i)}\in\ell_0^2(\N)^{\circ q_i}\,,\quad i\in\{1,\ldots,d\}\,,$$ whose entries are discrete multiple stochastic integrals, converges in distribution to a centred Gaussian random vector ${\bf N}$ with covariance matrix $\Sigma=(\sigma_{ij})_{i,j=1}^d$ if the following two conditions are satisfied:
\begin{itemize}
\item[(i)] $\lim\limits_{n\to\infty}\E[J_{q_i}(f_n^{(i)})J_{q_j}(f_n^{(j)})]=\lim\limits_{n\to\infty}{\rm Cov}(J_{q_i}(f_n^{(i)})J_{q_j}(f_n^{(j)}))=\sigma_{ij}$ for all $i,j\in\{1,\ldots,d\}$,
\item[(ii)]$\lim\limits_{n\to\infty}\|f_n^{(i)}\star_r^r f_n^{(i)}\|_{\ell^2(\N)^{\otimes 2(q_i-r)}}=0$ for all $i\in\{1,\ldots,d\}$ and all $r\in\{1,\ldots,q_i-1\}$.
\end{itemize}
In view of Remark \ref{rem:NormalGeneral}, (i) and (ii) imply that the sequence $(J_{q_i}(f_n^{(i)}))_{n\geq 1}$ satisfies a univariate central limit theorem in that $J_{q_i}(f_n^{(i)})$ converges in distribution to a one-dimensional centred Gaussian random variable with variance $\sigma_{ii}$. This is the discrete analogue to a similar phenomenon observed in \cite{NouPecRev,NuaOrt,PecTud,PecZhe} for Gaussian and Poisson multiple integrals, respectively.

\begin{remark}
A bound on the distance between the law of a vector consisting of discrete multiple stochastic integrals and a multivariate Gaussian random variable similar to that in Corollary \ref{cor:MutiIntegrals} above can also be deduced from universality results for so-called homogeneous sums as in \cite{NouPecReiInvariance}, see in particular Theorem 5.1 ibidem. However, this approach does not deliver the general estimate \eqref{eq:d4BoundGeneral}, which is of independent interest. For this reason, we prefer to give direct proofs using the multivariate Malliavin-Stein technique on the Rademacher chaos and not to employ universality results from \cite{NouPecReiInvariance}, which on their part are based, for example, on highly non-trivial arguments centred around the notion of influence, see \cite{Influence}. We use an interpolation technique, which has already been applied in the Gaussian and Poisson context (cf. \cites{NouPecReiInvariance,PecZhe}) together with a new multivariate approximate chain rule for the gradient operator (see Lemma \ref{lem:Multi1}), which generalizes the one-dimensional chain rule in \cite{ReiPec}*{Proposition 2.14}.
\end{remark}

\begin{remark}\label{MultivariateKolmogorov}
In contrast to the one-dimensional case, in this section a probability metric based on smooth test functions is considered, namely the $d_4$-distance. The multivariate Kolmogorov distance can then be estimated from above in terms of the $d_4$-distance using a smoothing argument (see Lemma 12.1 in \cite{CheGolSha} and the references cited there), which usually leads to suboptimal rates of convergence. Dealing with the multivariate Kolmogorov distance without smoothing techniques would require precise information on the solution of the multivariate Stein equation associated with a multivariate indicator function. To the best of our knowledge, this is still an open problem in the theory of Stein's method. In this context we emphasize that the multivariate normal approximation of a vector of multiple stochastic integrals on the Gaussian (cf.\ \cite{NouPecRev}) or the Poisson space (cf.\ \cite{PecZhe}) in terms of the multivariate Kolmogorov distance has not been considered for the same reason.
\end{remark}

To give a proof of Theorem \ref{thm:MultiGeneral} we need the following two lemmas. Recall that by $X=(X_k)_{k\in\N}$ we denote a Rademacher sequence and that a Rademacher functional $F=F(X)$ is a (possibly) non-linear transformation of $X$.

\begin{lemma}\label{lem:Multi1}
Fix an integer $d\geq 2$, let $F_1,\ldots,F_d\in{\rm dom}(D)$ be Rademacher functionals satisfying $\E[F_i]=0$ for all $i\in\{1,\ldots,d\}$, and define the random vector ${\bf F}:=(F_1,\ldots,F_d)$. Let $f:\R^d\rightarrow\R$ be thrice differentiable with continuous and bounded partial derivatives up to order three. Then, for every $k\in\N$,
\begin{align*}
D_kf(\textbf{F}) &=\sum_{i=1}^d\frac{\partial}{\partial x_i}f(\textbf{F})(D_kF_i)\\ 
&\qquad-\frac{1}{2}\sum_{i,j=1}^d\Big(\frac{\partial^2}{\partial x_i\partial x_j}f({\bf F}_k^+)-\frac{\partial^2}{\partial x_i\partial x_j}f({\bf F}_k^-)\Big)(D_kF_i)(D_kF_j)X_k+R\,,
\end{align*}
where the remainder term $R$ satisfies
\[R=\sum_{i,j,l=1}^dR_{i,j,l}\quad\text{ with }\quad |R_{i,j,l}|\leq\frac{10}{3}\sup_{x\in\R^d}\left|\frac{\partial^3}{\partial x_i\partial x_j\partial x_l}f(x)\right|\,|(D_kF_i)(D_kF_j)(D_kF_l)|\,.\]
\end{lemma}
\begin{proof}
For $k\in\N$ we use the definition of the discrete gradient together with a Taylor series expansion of $f$ to see that
\begin{align*}
&D_kf(\textbf{F})\\
&=\frac{1}{2}\left(f(\textbf{F}_k^+)-f(\textbf{F}_k^-)\right)\\
&=\frac{1}{2}\left(f(\textbf{F}_k^+)-f(\bf{F})\right)-\frac{1}{2}\left(f(\textbf{F}_k^-)-f(\textbf{F})\right)\\
&=\frac{1}{2}\sum_{i=1}^d\frac{\partial}{\partial x_i}f(\textbf{F})((F_i)_k^+-F_i)+\frac{1}{4}\sum_{i,j=1}^d\frac{\partial^2}{\partial x_i\partial x_j}f(\textbf{F})((F_i)_k^+-F_i)((F_j)_k^+-F_j)\\
&\qquad-\Big(\frac{1}{2}\sum_{i=1}^d\frac{\partial}{\partial x_i}f(\textbf{F})((F_i)_k^--F_i)+\frac{1}{4}\sum_{i,j=1}^d\frac{\partial^2}{\partial x_i\partial x_j}f(\textbf{F})((F_i)_k^--F_i)((F_j)_k^--F_j)\Big)\\
&\qquad +R_1+R_2
\end{align*}
\begin{align*}
&=\sum_{i=1}^d\frac{\partial}{\partial x_i}f(\textbf{F}) D_kF_i\\
&\qquad +\frac{1}{8}\sum_{i,j=1}^d\frac{\partial^2}{\partial x_i\partial x_j}f(\textbf{F})\left[((F_i)_k^+-F_i)((F_j)_k^+-F_j)-((F_i)_k^--F_i)((F_j)_k^--F_j)\right]\\
&\qquad +\frac{1}{8}\sum_{i,j=1}^d\frac{\partial^2}{\partial x_i\partial x_j}f(\textbf{F})\left[((F_i)_k^+-F_i)((F_j)_k^+-F_j)-((F_i)_k^--F_i)((F_j)_k^--F_j)\right]\\
&\qquad + R_1+R_2\,.
\end{align*}
where 
\[R_1:=\sum_{i,j,l=1}^d R_{i,j,l}^{(1)}\]
with 
\begin{align*}
|R_{i,j,l}^{(1)}|&\leq\frac{1}{12}\sup_{x\in\R^d}\left|\frac{\partial^3}{\partial x_i\partial x_j\partial x_l}f(x)\right|\,|((F_i)_k^+-F_i)((F_j)_k^+-F_j)((F_l)_k^+-F_l)|\\
&\leq\frac{2}{3}\sup_{x\in\R^d}\left|\frac{\partial^3}{\partial x_i\partial x_j\partial x_l}f(x)\right|\,\left|(D_kF_i)(D_kF_j)(D_kF_l)\right|\,,
\end{align*}
and
\[R_2:=\sum_{i,j,l=1}^d R_{i,j,l}^{(2)}\]
satisfies 
\begin{align*}
|R_{i,j,l}^{(2)}|&\leq\frac{1}{12}\sup_{x\in\R^d}\left|\frac{\partial^3}{\partial x_i\partial x_j\partial x_l}f(x)\right|\,|((F_i)_k^--F_i)((F_j)_k^--F_j)((F_l)_k^--F_l)|\\
&\leq\frac{2}{3}\sup_{x\in\R^d}\left|\frac{\partial^3}{\partial x_i\partial x_j\partial x_l}f(x)\right|\,\left|(D_kF_i)(D_kF_j)(D_kF_l)\right|\,.
\end{align*}
To the two sums on the right hand side of the last equation we now apply a Taylor expansion of $\frac{\partial^2}{\partial x_i\partial x_j}f$ up to order $1$ about ${\bf F}_k^+$ and ${\bf F}_k^-$, respectively. We thus obtain that
\begin{align*}
D_kf(\textbf{F})
&=\sum_{i=1}^d\frac{\partial}{\partial x_i}f(\textbf{F}) D_kF_i+\frac{1}{8}\sum_{i,j=1}^d\bigg\{\left[\frac{\partial^2}{\partial x_i\partial x_j}f(\textbf{F}_k^+)+\frac{\partial^2}{\partial x_i\partial x_j}f(\textbf{F}_k^-)\right]\\
&\qquad\times\left[((F_i)_k^+-F_i)((F_j)_k^+-F_j)-((F_i)_k^--F_i)((F_j)_k^--F_j)\right]\bigg\} + R_1+R_2+R_3\,,
\end{align*}
where 
\[R_3:=\sum_{i,j,l=1}^d R_{i,j,l}^{(3)}\]
is such that
\begin{align*}
&\quad|R_{i,j,l}^{(3)}|\\
&\leq\frac{1}{8}\sup_{x\in\R^d}\left|\frac{\partial^3}{\partial x_i\partial x_j\partial x_l}f(x)\right|\\
&\quad\times\left(|((F_i)_k^+-F_i)((F_j)_k^+-F_j)((F_l)_k^+-F_l)\right|+\left|((F_i)_k^--F_i)((F_j)_k^--F_j)((F_l)_k^--F_l)|\right)\\
&\leq 2\sup_{x\in\R^d}\left|\frac{\partial^3}{\partial x_i\partial x_j\partial x_l}f(x)\right|\,\left|(D_kF_i)(D_kF_j)(D_kF_l)\right|\,.
\end{align*}
This yields the result.
\end{proof}

\begin{lemma}\label{lem:Multi2}
Fix an integer $d\geq 2$, let $F_0,F_1,\ldots,F_d\in{\rm dom}(D)$ be Rademacher functionals satisfying $\E[F_i]=0$ and $\E[\|DF_i\|^4_{\ell^4(\N)}]<\infty$ for $1\leq i\leq d$ and define the random vector ${\bf F}:=(F_1,\ldots,F_d)$. Let $f:\R^d\rightarrow\R$ be thrice differentiable with continuous and bounded partial derivatives up to order three. Then,
\begin{align*}
\E[f({\bf F})\,F_0] = \E\Big[\sum_{j=1}^d{\partial \over \partial x_j}f({\bf F})\,\langle DF_j,-DL^{-1}F_0\rangle_{\ell^2(\N)}\Big] + \E[\langle R,-DL^{-1}F_0\rangle_{\ell^2(\N)}]
\end{align*}
with $R$ satisfying the estimate
\begin{align*}
\big|\E[\langle R,-DL^{-1}F_0\rangle_{\ell^2(\N)}]\big|\leq{10\over 3}M_3(f)\E\big[\big\langle \big(\sum_{j=1}^d|DF_j|\big)^3,|DF^{-1}F_0|\big\rangle_{\ell^2(\N)}\big]\,.
\end{align*}
\end{lemma}
\begin{proof}
We use relation \eqref{eq:L=-dD} and the integration-by-parts formula to see that
\begin{align*}
\E[f({\bf F})\,F_0] = \E[\delta(-DL^{-1}F_0)\,f({\bf F})] = \E[\langle Df({\bf F}),-DL^{-1}F_0\rangle_{\ell^2(\N)}]\,,
\end{align*}
since $f({\bf F})\in{\rm dom}(D)$ due to the assumed boundedness of the partial derivatives of the test function $f$.
Now, we apply Lemma \ref{lem:Multi1} to re-write $Df({\bf F})$. This leads to
\begin{align*}
\E[f({\bf F})\,F_0] = \E\Big[\sum_{j=1}^d{\partial\over \partial x_j}f({\bf F})\langle DF_j,-DL^{-1}F_0\rangle_{\ell^2(\N)}\Big]+\E[\langle R,-DL^{-1}F_0\rangle_{\ell^2(\N)}]
\end{align*}
with the term $R$ satisfying
$$\big|\E[\langle R,-DL^{-1}F_0\rangle_{\ell^2(\N)}]\big|\leq{10\over 3}M_3(f)\E\big[\big\langle\big(\sum_{j=1}^d|DF_j|\big)^3,|DL^{-1}F_0|\big\rangle_{\ell^2(\N)}\big]\,.$$ Note that the term involving second-order partial derivatives of $f$ vanishes because of \cite[Lemma 2.13 (1)]{ReiPec} and an application of Fubini's theorem. More precisely, our assumption that $\E[\|DF_i\|^4_{\ell^4(\N)}]<\infty$ for all $1\leq i\leq d$ and \cite[Lemma 2.13 (3)]{ReiPec} is needed in order to justify an exchange of the order of integration by means of Fubini's theorem. This proves the claim.
\end{proof}

\begin{proof}[Proof of Theorem \ref{thm:MultiGeneral}]
We use the so-called smart-path technique, which in the context of the Malliavin-Stein method has previously found application in \cite{NouPecReiInvariance} and \cite{PecZhe}.

Let $g:\R^d\to\R$ be continuously differentiable up to order $4$ and define $$\Psi(t):=\E[g(\sqrt{t}\,{\bf N}+\sqrt{1-t}\,{\bf F})]\,.$$ Here ${\bf F}=(F_1,\ldots,F_d)$ is a vector of Rademacher functionals and ${\bf N}=(N_1,\ldots,N_d)$ is a Gaussian random vector with covariance matrix $\Sigma=(\sigma_{ij})_{i,j=1}^d$. Then,
$$\big|\E[g({\bf F})]-\E[g({\bf N})]\big|=\big|\Psi(0)-\Psi(1)\big|\leq\sup_{t\in(0,1)}\big|\Psi'(t)\big|$$
with the derivative of $\Psi$ given by
\begin{align}\label{eq:PsiPrime}
\Psi'(t)={1\over 2\sqrt{t}}A-{1\over 2\sqrt{1-t}}B\,,
\end{align}
where
\begin{align*}
A &:=\sum_{i=1}^d\E\Big[{\partial\over \partial x_i}g(\sqrt{t}\,{\bf N}+\sqrt{1-t}\,{\bf F})\,N_i\Big]\quad\text{and}\quad
B := \sum_{i=1}^d\E\Big[{\partial\over \partial x_i}g(\sqrt{t}\,{\bf N}+\sqrt{1-t}\,{\bf F})\,F_i\Big]\,.
\end{align*}
As in the proof of Theorem 4.2 in \cite{PecZhe} an integration-by-parts argument shows that
\begin{equation}\label{eq:AUmgeformt}
A = \sqrt{t}\,\sum_{i,j=1}^d\sigma_{ij}\,\E\Big[ {\partial^2\over\partial x_i\partial x_j}g(\sqrt{t}\,{\bf N}+\sqrt{1-t}\,{\bf F})\Big]\,.
\end{equation}
To evaluate $B$ further, let us condition on ${\bf N}={\bf b}\in\R^d$, i.e., 
\begin{align*}
B &= \sum_{i=1}^d\E\Big[\E\Big[{\partial\over \partial x_i}g(\sqrt{t}\,{\bf b}+\sqrt{1-t}\,{\bf F})\,F_i\Big]_{{\bf b}={\bf N}}\Big]
\end{align*}
and put $g_i^{t,{\bf b}}({\bf F}):={\partial\over \partial x_i}g(\sqrt{t}\,{\bf b}+\sqrt{1-t}\,{\bf F}).$
Now, we apply Lemma \ref{lem:Multi2} to the conditional expectation, which leads to
\begin{align*}
\E\Big[g_i^{t,{\bf b}}({\bf F})\,F_i\Big]_{{\bf b}={\bf N}} &=\E\Big[\sum_{j=1}^d{\partial\over \partial x_j}g_i^{t,{\bf b}}({\bf F})\,\langle DF_j,-DL^{-1}F_i\rangle_{\ell^2(\N)}\Big]_{{\bf b}={\bf N}}\\
&\qquad\qquad\qquad\qquad\qquad+\E[\langle R,-DL^{-1}F_i\rangle_{\ell^2(\N)}]_{{\bf b}={\bf N}}\\
&=\sqrt{1-t}\,\E\Big[\sum_{j=1}^d{\partial\over \partial x_j\partial x_i}g(\sqrt{t}\,{\bf b}+\sqrt{1-t}\,{\bf F})\,\langle DF_j,-DL^{-1}F_i\rangle_{\ell^2(\N)}\Big]_{{\bf b}={\bf N}}\\
&\qquad\qquad\qquad\qquad\qquad+\E[\langle R,-DL^{-1}F_i\rangle_{\ell^2(\N)}]_{{\bf b}={\bf N}}\,,
\end{align*}
with $R$ satisfying
$$\big|\E[\langle R,-DL^{-1}F_i\rangle_{\ell^2(\N)}]_{{\bf b}={\bf N}}\big|\leq {10\over 3}M_3(g_i^{t,{\bf b}})\E\big[\big\langle\big(\sum_{j=1}^d|DF_j|\big)^3,|DL^{-1}F_i|\big\rangle_{\ell^2(\N)}\big]\,.$$
Thus,
\begin{align*}
B = & \sqrt{1-t}\,\sum_{i,j=1}^d\E\Big[{\partial^2\over \partial x_i\partial x_j}g(\sqrt{t}\,{\bf N}+\sqrt{1-t}\,{\bf F})\,\langle DF_j,-DL^{-1}F_i\rangle_{\ell^2(\N)}\Big]\\
&\qquad\qquad+\sum_{i=1}^d\E[\E[\langle R,-DL^{-1}F_i\rangle_{\ell^2(\N)}]_{{\bf b}={\bf N}}]\,.
\end{align*}
Now, note that $M_3(g_i^{t,{\bf b}})\leq (1-t)^{3/2}M_4(g)$ so that
\begin{align*}
&\Big|\sum_{i=1}^d\E[\E[\langle R,-DL^{-1}F_i\rangle_{\ell^2(\N)}]_{{\bf b}={\bf N}}]\Big|\\
&\qquad\qquad\leq {10\over 3}M_4(g)(1-t)^{3/2}\E\big[\big\langle\big(\sum_{j=1}^d|DF_j|\big)^3,\sum_{i=1}^d|DL^{-1}F_i|\big\rangle_{\ell^2(\N)}\big]\,,
\end{align*}
independently of $\bf b$. Thus, \eqref{eq:PsiPrime} implies the bound
\begin{align*}
\big|\E[g({\bf F})]-\E[g({\bf N})]\big| &\leq {1\over 2}M_2(g)\sum_{i,j=1}^d\E[|\sigma_{ij}-\langle DF_j,-DL^{-1}F_i\rangle_{\ell^2(\N)}|]\\
&\qquad +{1\over 2}{10\over 3}{1\over \sqrt{1-t}}M_4(g)(1-t)^{3/2}\E\big[\big\langle\big(\sum_{j=1}^d|DF_j|\big)^3,\sum_{i=1}^d|DL^{-1}F_i|\big\rangle_{\ell^2(\N)}\big]\\
&\leq {d\over 2}\Big(\sum_{i,j=1}^d\E\Big[(\sigma_{ij}-\langle DF_j,-DL^{-1}F_i\rangle_{\ell^2(\N)})^2\Big]\Big)^{1/2}\\
&\qquad +{5\over 3}\E\big[\big\langle\big(\sum_{j=1}^d|DF_j|)^3,\sum_{i=1}^d|DL^{-1}F_i|\big\rangle_{\ell^2(\N)}\big]\,,
\end{align*}
where in the last line we have applied the Cauchy-Schwarz inequality and took the supremum over all $t\in(0,1)$ and over all functions $g:\R^d\to\R$ with $M_2(g)\leq 1$ and $M_4(g)\leq 1$. This proves the theorem.
\end{proof}

\begin{proof}[Proof of Corollary \ref{cor:MutiIntegrals}]
As in the proof of Theorem \ref{BerryEsseenMultipleIntegrals} we assume without loss of generality that the multiple integrals $J_{q_i}(f^{(i)})$ have kernels $f^{(i)}$ with finite support for all $1\leq i\leq d$.
Fix $i,j\in\{1,\dots,d\}$ and to ease the notation let $\sigma:=\sigma_{ij}$, $p:=q_j$, $q:=q_i$, $f:=f^ {(j)}$, $g:=f^{(i)}$. Assume without loss of generality that $p\leq q$. We apply Theorem \ref{thm:MultiGeneral} to the random vector $\textbf{F}$
(note that the assumption $\E[\|DF_i\|^4_{\ell^4(\N)}]<\infty$ for $1\leq i\leq d$ is fulfilled by Lemma \ref{Hypercontractivity}, see also \cite{ReiPec}). Using the multiplication formula \eqref{Multiplication formula} we get for all $k\in\N$,
\begin{align*}
&(D_kJ_p(f))(-D_kL^{-1}J_q(g))\\
&=\frac{1}{q}(D_kJ_p(f))(D_kJ_q(g))\\
&=pJ_{p-1}(f(\,\cdot\,,k))J_{q-1}(g(\,\cdot\,,k))\\
&=p\sum_{r=0}^{p-1}r!{p-1\choose r}{q-1\choose r}J_{p+q-2(r+1)}\big(\widetilde{\big(f(\,\cdot\,,k)\star_r^r g(\,\cdot\,,k)\big)}\1_{\Delta_{p+q-2(r+1)}}\big)\,.
\end{align*}
Thus,
\begin{align*}
&\langle DJ_p(f),-DL^{-1}J_q(g)\rangle_{\ell^2(\N)}\\
&=p\sum_{r=0}^{p-1}r!{p-1\choose r}{q-1\choose r}J_{p+q-2(r+1)}\big(\widetilde{\big(f\star_{r+1}^{r+1} g\big)}\1_{\Delta_{p+q-2(r+1)}}\big)\\
&=p\sum_{r=1}^{p}(r-1)!{p-1\choose r-1}{q-1\choose r-1}J_{p+q-2r}\big(\widetilde{\big(f\star_{r}^{r} g\big)}\1_{\Delta_{ p+q-2r}}\big)\,.
\end{align*}
If $p<q$ we get by isometry \eqref{Isometry property} of discrete multiple stochastic integrals,
\begin{align}
&\E\big[\big(\sigma-\langle DJ_p(f),-DL^{-1}J_q(g)\rangle\big)^2\big]\notag\\
&=\sigma^2+p^2\sum_{r=1}^{p}((r-1)!)^2{p-1\choose r-1}^2{q-1\choose r-1}^2(p+q-2r)!\lnormb{2}{p+q-2r}{\widetilde{\big(f\star_{r}^{r} g\big)}\1_{\Delta_{p+q-2r}}}^2\\
&\leq \sigma^2+p^2\sum_{r=1}^{p}((r-1)!)^2{p-1\choose r-1}^2{q-1\choose r-1}^2(p+q-2r)!\lnormb{2}{p+q-2r}{f\star_{r}^{r} g}^2.\label{eq:p<q}
\end{align}
If $p=q$ we get
\begin{align}
&\E\big[\big(\sigma-\langle DJ_p(f),-DL^{-1}J_q(g)\rangle_{\ell^2(\N)}\big)^2\big]\notag\\
&=\E\big[\big((\sigma-p!\langle f,g\rangle_{\ell^2(\N)^{\otimes p}})-p\sum_{r=1}^{p-1}(r-1)!{p-1\choose r-1}^2J_{2(p-r)}\big(\widetilde{\big(f\star_{r}^{r} g\big)}\1_{\Delta_{2(p-r)}}\big)\big)^2\big]\notag\\
&=(\sigma-p!\langle f,g\rangle_{\ell^2(\N)^{\otimes p}})^2\notag\\
&\qquad+p^2\sum_{r=1}^{p-1}((r-1)!)^2{p-1\choose r-1}^4(2(p-1))!\lnormb{2}{2(p-r)}{\widetilde{\big(f\star_{r}^{r} g\big)}\1_{\Delta_{2(p-r)}}}^2\notag\\
&\leq (\sigma-\E[J_p(f)J_q(g)])^2\notag\\
&\qquad+p^2\sum_{r=1}^{p-1}((r-1)!)^2{p-1\choose r-1}^4(2(p-1))!\lnormb{2}{2(p-r)}{f\star_{r}^{r} g}^2.\label{eq:p=q}
\end{align}
Let us now consider the second term in the bound of Theorem \ref{thm:MultiGeneral}. Let $F_i:=J_{q_i}(f^{(i)})$ for $1\leq i\leq d$. We have 
\begin{align}
\E\big[\big\langle\big(\sum_{j=1}^d|DF_j|\big)^3,\sum_{i=1}^d|DL^{-1}F_i|\big\rangle_{\ell^2(\N)}\big]&=\E\big[\big\langle\big(\sum_{j=1}^d|DF_j|\big)^3,\sum_{i=1}^d\frac{1}{q_i}|DF_i|\big\rangle_{\ell^2(\N)}\big]\notag\\
&\leq\frac{1}{\min\{q_1,\dots,q_d\}}\E\big[\sum_{k=1}^{\infty}\big(\sum_{j=1}^d|D_kF_i|\big)^4\big]\notag\\
&\leq\frac{d^3}{\min\{q_1,\dots,q_d\}}\sum_{i=1}^d\E[\|DF_i\|^4_{\ell^4(\N)}]\label{eq:hoelder}\\
&\leq C\!\max\limits_{{i=1,\dots,d}\atop {r=1\dots,q_i-1}}\big\{\lnormb{2}{2(q_i-r)}{f^{(i)}\star_r^rf^{(i)}}^2\big\}\label{eq:Term2},
\end{align}
with a constant $0<C<\infty$ only depending on $d$ and on $q_1,\ldots,q_d$. In \eqref{eq:hoelder} H\"older's inequality with H\"older conjugates $4$ and $\frac{4}{3}$ has been used. Inequality \eqref{eq:Term2} follows from \eqref{eq:4NormHoch4}. Now, Theorem \ref{thm:MultiGeneral} together with \eqref{eq:p<q}, \eqref{eq:p=q} and \eqref{eq:Term2} gives the first inequality in Corollary \ref{cor:MutiIntegrals}.
To derive the second inequality, we need to show that all `mixed' contractions $f\star_r^rg=f^{(i)}\star_r^r f^{(j)}$  appearing in \eqref{eq:p<q} and \eqref{eq:p=q} can be bounded by contractions of the form $f\star_r^rf$ and $g\star_r^rg$ build up either by $f$ or $g$. In the following we use the vector notation ${\bf k}_r:=(k_1,\dots,k_r)\in\N^r$ for summation indices. Moreover ${\bf k}_r{\bf k}_\ell$ stands for the concatenation $(k_1,\ldots,k_{r},k_{r+1},\ldots,k_{r+\ell})$ of two vectors ${\bf k}_r$ and ${\bf k}_\ell$. First, let us consider the case $r<p\leq q$. Then,
\begin{align}
&\lnormb{2}{p+q-2r}{f\star_r^r g}^2=\sum_{\textbf{k}_{p+q-2r}}(f\star_r^r g)({\bf k}_{p+q-2r})^2\notag\\
&=\sum_{\textbf{k}_{p-r}}\sum_{\textbf{k}_{q-r}}\sum_{\textbf{a}_r}\sum_{\textbf{b}_r}f(\textbf{k}_{p-r}\textbf{a}_r)f(\textbf{k}_{p-r}\textbf{b}_r)g(\textbf{k}_{q-r}\textbf{a}_r)g(\textbf{k}_{q-r}\textbf{b}_r)\notag\\
&=\sum_{\textbf{a}_r}\sum_{\textbf{b}_r}(f\star_{p-r}^{p-r}f)(\textbf{a}_r\textbf{b}_r)\,(g\star_{q-r}^{q-r}g)(\textbf{a}_r\textbf{b}_r)\notag\\
&=\big\langle f\star_{p-r}^{p-r}f,g\star_{q-r}^{q-r}g\big\rangle_{\ell^2(\N)^{\otimes 2r}}\leq \lnormb{2}{2r}{f\star_{p-r}^{p-r}f}\,\lnormb{2}{2r}{g\star_{q-r}^{q-r}g},\notag
\end{align}
which implies that 
\begin{align}
\lnormb{2}{p+q-2r}{f\star_r^r g}\leq\max\big\{\lnormb{2}{2r}{f\star_{p-r}^{p-r}f}\, ,\,\lnormb{2}{2r}{g\star_{q-r}^{q-r}g}\big\}.\label{eq:Absch1}
\end{align}
If $r=p<q$, we get
\begin{align}
&\lnorm{2}{(q-p)}{f\star_p^p g}^2
=\sum_{\textbf{k}_{q-p}}(f\star_p^pg)(\textbf{k}_{q-p})^2\notag\\
&=\sum_{\textbf{k}_{q-p}}\sum_{\textbf{a}_{r}}\sum_{\textbf{b}_{r}}f(\textbf{a}_{r})f(\textbf{b}_{r})g(\textbf{k}_{q-p}\textbf{a}_{r})g(\textbf{k}_{q-p}\textbf{b}_{r})\notag\\
&=\sum_{\textbf{a}_{r}}\sum_{\textbf{b}_{r}}f(\textbf{a}_{r})f(\textbf{b}_{r})(g\star_{q-p}^{q-p}g)(\textbf{a}_{r}\textbf{b}_{r})\,,\notag
\end{align}
Hence,
\begin{equation}\label{eq:Absch2}
\lnormb{2}{(q-p)}{f\star_p^p g}\leq \lnormb{2}{p}{f}\big(\lnormb{2}{2p}{g\star_{q-p}^{q-p}g}\big)^ {1/2}\,.
\end{equation}
Now, the second estimate \eqref{eq:Cor2} in Corollary \ref{cor:MutiIntegrals} follows from the first estimate \eqref{eq:Cor1} by means of \eqref{eq:Absch1} and \eqref{eq:Absch2}. This completes the proof.


\end{proof}

\section{Applications}\label{sec:applications}

\subsection{Infinite weighted $2$-runs}

The notion of a $2$-run is one of the most simple dependency structures and has been studied exhaustively in the literature, see \cite{BalKou}. For example, \cite{SteinProc} uses the so-called local approach, while \cite{ReinertRoellin,RinottRotar} apply exchangeable pair coupling constructions. The first Berry-Esseen bound for the number of finite $2$-runs appeared in \cite{Godbol}. In contrast to most of the previously available results, our method, which is based on a discrete version of the Malliavin calculus of variations, allows to treat infinite weighted $2$-runs. This continues the line of research initiated in Section 5 of \cite{ReiPec} and provides, up to our best knowledge, the first Berry-Esseen bound for such functionals.

\medspace

Let $X=(X_i)_{i\in\Z}$ be a double-sided sequence of i.i.d.\ Rademacher random variables and let for each $n\in\N$, $(a_i^{(n)})_{i\in\Z}$ be a double-sided sequence of real numbers. The sequence $(F_n)_{n\in\N}$ of normalized infinite weighted $2$-runs is then defined as
$$F_n:={G_n-\E G_n\over \sqrt{\Var G_n}}\,,\qquad G_n:=\sum_{i\in\Z}a_i^{(n)}\,Y_iY_{i+1}\,,\qquad n\in\N,$$
where $Y_i:=\frac{1}{2}\big(1-X_i\big)$ for $i\in\Z$. In other words, $G_n$ counts the weighted number of subsequences of $1$'s of length two in an infinite double-sided sequence of Bernoulli trials. We notice that $$\Var G_n={3\over 16}\sum_{i\in\Z}(a_i^{(n)})^2+{1\over 8}\sum_{i\in\Z}a_i^{(n)}a_{i+1}^{(n)}\,.$$ Our result reads as follows.

\begin{theorem}
Let $(F_n)_{n\in\N}$ be a sequence of normalized infinite weighted $2$-runs as above and let $N$ be a standard Gaussian random variable. Then, $$d_K(F_n,N)\leq C\max\Big\{(\Var G_n)^{-3/2}\sum_{i\in\Z}|a_i^{(n)}|^3,(\Var G_n)^{-1}\Big(\sum_{i\in\Z}(a_i^{(n)})^4\Big)^{1/2}\,\Big\}$$
with a constant $0<C<\infty$ not depending on $n$.
\end{theorem}
\begin{proof}
We notice first that each $F_n$ has chaotic decomposition $F_n=J_1(f_n^{(1)})+J_2(f_n^{(2)})$ with $f_n^{(1)}$ and $f_n^{(2)}$ given by
\begin{align*}
f_n^{(1)}(k) &= {1\over 4\sqrt{\Var G_n}}\sum_{i\in\Z}a_i^{(n)}\big(\1_{\{k=i\}}+\1_{\{k=i+1\}}\big)\,,\qquad k\in\Z\,,\\
f_n^{(2)}(k,l) &={1\over 8\sqrt{\Var G_n}}\sum_{i\in\Z}a_i^{(n)}\big(\1_{\{k=i,l=i+1\}}+\1_{\{k=i+1,l=i\}}\big)\,,\qquad k,l\in\Z\,.
\end{align*}
Thus,
\begin{align*}
(f_n^{(1)}\star_1^1 f_n^{(2)})(i) &= {1\over 32\Var G_n}\big(a_i^{(n)}a_{i+1}^{(n)}+(a_{i-1}^{(n)})^2+\big(a_i^{(n)}\big)^2+a_{i-1}^ {(n)}a_{i-2}^{(n)}\big)\,,\\
(f_n^{(2)}\star_1^1 f_n^{(2)})(i,j) &={1\over 64\Var G_n}\big(a_i^{(n)}a_{i+1}^{(n)}\1_{\{j=i+2\}}+a_j^{(n)}a_{j+1}^{(n)}\1_{\{j=i-2\}}\big)\,,\quad i\neq j\,,\\
(f_n^{(2)}\star_1^1f_n^{(2)})(i,i)&= {1\over 64\Var G_n}\big((a_{i-1}^{(n)})^2+(a_i^{(n)})^2\big) \,,
\end{align*}
and consequently
\begin{align*}
\|f_n^{(1)}\star_1^1 f_n^{(2)}\|_{\ell^2(\Z)} &\leq {\sqrt{2}\over 16\Var G_n}\Big(\sum_{i\in\Z}(a_i^{(n)})^2(a_{i+1}^{(n)})^2+\sum_{i\in\Z}(a_i^{(n)})^4\Big)^{1/2}\\
&\leq {1\over 8\Var G_n}\Big(\sum_{i\in\Z}(a_i^{(n)})^4\Big)^{1/2}\,,\\
\|(f_n^{(2)}\star_1^1 f_n^{(2)})\1_{\Delta_2}\|_{\ell^2(\Z)^{\otimes 2}} & \leq {\sqrt{2}\over 64\Var G_n}\Big(\sum_{i\in\Z}(a_i^{(n)})^2(a_{i+1}^{(n)})^2\Big)^{1/2}\\
&\leq {\sqrt{2}\over 64\Var G_n}\Big(\sum_{i\in\Z}(a_i^{(n)})^4\Big)^{1/2}\,,\\
\|(f_n^{(2)}\star_1^1 f_n^{(2)})\1_{\Delta_2^c}\|_{\ell^2(\Z)^{\otimes 2}} & \leq {1\over 64\Var G_n}\bigg(\sum_{i\in\Z}\big((a_{i-1}^{(n)})^2+(a_i^{(n)})^2\big)^2\bigg)^{1/2} \\
&\leq {\sqrt{2}\over 64\Var G_n}\bigg(\sum_{i\in\Z}\big((a_{i-1}^{(n)})^4+(a_i^{(n)})^4\big)\bigg)^{1/2}\\
&\leq {1\over 32\Var G_n}\bigg(\sum_{i\in\Z}(a_i^{(n)})^4\bigg)^{1/2}\,,
\end{align*}
where we have used the Cauchy-Schwarz inequality.
Moreover,
\begin{align*}
\|f_n^{(1)}\|_{\ell^4(\Z)}^2 \leq {1\over 4 \Var G_n}\bigg(\sum_{i\in\Z}(a_i^{(n)})^4\bigg)^{1/2}\,.
\end{align*}

Multiple use of the Cauchy-Schwarz inequality gives
\begin{align*}
\sum_{k,j\in\Z}(f_n^{(1)}(k))^2(f_n^{(2)}(k,j))^2 &= {1\over 4096(\Var G_n)^2}\sum_{i\in\Z}\big((a_{i-1}^{(n)})^2+(a_i^{(n)})^2\big)\big(a_{i-1}^{(n)}+a_i^{(n)}\big)^2\\
&\leq  {1\over 2048(\Var G_n)^2}\sum_{i\in\Z}\big((a_{i-1}^{(n)})^2+(a_i^{(n)})^2\big)^2\\
&\leq  {1\over 1024(\Var G_n)^2}\sum_{i\in\Z}\big((a_{i-1}^{(n)})^4+(a_i^{(n)})^4\big)\\
& \leq {1\over 512(\Var G_n)^2}\sum_{i\in\Z}(a_i^{(n)})^4\,.
\end{align*}
Hence,
\begin{align*}
\bigg(\sum_{k,j\in\Z}(f_n^{(1)}(k))^2(f_n^{(2)}(k,j))^2\bigg)^{1/2} &\leq\frac{1}{\sqrt{512}\Var G_n}\bigg(\sum_{i\in\Z}(a_i^{(n)})^4\bigg)^{1/2}.
\end{align*}
Finally, by H\"older's inequality, we have
\begin{align*}
\sum_{k\in\Z}\Big(|f^{(1)}(k)|+2\sum_{j\in\Z}|f^{(2)}(j,k)|\Big)^3 &\leq 4\sum_{k\in\Z}\Big(|f^{(1)}(k)|^3+8\Big(\sum_{j\in\Z}|f^{(2)}(j,k)|\Big)^3\Big)\\
&=4\sum_{k\in\Z}|f^{(1)}(k)|^3+32\sum_{k\in\Z}\Big(\sum_{j\in\Z}|f^{(2)}(j,k)|\Big)^3\,,
\end{align*}
where each of these summands is bounded by a constant times $${1\over (\Var G_n)^{3/2}}\sum_{i\in\Z}|a_i^{(n)}|^3\,.$$ Now, applying Theorem \ref{thm:SumSingle+Double} proves the result.
\end{proof}

\subsection{A combinatorial central limit theorem}

Our second application deals with an extended version of a combinatorial central limit theorem of Blei and Janson \cite{BleJan}, which has also been studied in \cite{ReiPec}. Using Theorem \ref{BerryEsseenMultipleIntegrals} we can strengthen these results to a Berry-Esseen bound without imposing further conditions. 

The general set-up is as follows. Fix $q\geq 2$ and let $F\subseteq\Delta_q$ be a (possibly infinite) non-empty subset of $\N^q$.
Let further $b=(b_i)_{i\geq 1}\in\ell^2(\N)$ and define a measure $\mu_b$ on $\N$ by putting $$\mu_b(B):=\sum_{i\in B}b_i^2\,,\qquad B\subset\N\,.$$ In what follows we assume that $\mu_b^{\otimes q}(F)>0$, where $\mu_b^{\otimes q}$ stands for the $q$-fold product measure of $\mu_b$. Given a Rademacher sequence $X=(X_i)_{i\geq 1}$, we now define the random variable $S^{(b)}(F)$ by
\begin{equation}\label{eq:DefSbF}
S^{(b)}(F):={1\over (q!\mu_b^{\otimes q}(F))^{1/2}}\sum_{(i_1,\ldots,i_q)\in F}b_{i_1}\cdots b_{i_q}\,X_{i_1}\cdots X_{i_q}\,.
\end{equation}
To discuss the distance between $S^{(b)}(F)$ and a standard Gaussian random variable we need further notation. By $F_j^*$ we denote the collection of all $(i_1,\ldots,i_q)\in F$ such that $i_k=j$ for some $k\in\{1,\ldots,q\}$. Further, let $F^\sharp\subset F\times F$ be defined as follows. A pair $\big((i_1,\ldots,i_q),(j_1,\ldots,j_q)\big)$ belongs to $F^\sharp$ if $\{i_1,\ldots,i_q\}\cap\{j_1,\ldots,j_q\}=\emptyset$ and there are $(k_1,\ldots,k_q),(l_1,\ldots,l_q)\in F$ such that $\{k_1,\ldots,k_q,l_1,\ldots,l_q\}=\{i_1,\ldots,i_q,j_1,\ldots,j_q\}$ with the property that $(k_1,\ldots,k_q)$ does not coincide with $(i_1,\ldots,i_q)$ or $(j_1,\ldots,j_q)$. We finally define the quantities $\Phi^{(b)}(F)$ and $\Psi_j^{(b)}(F)$ by
\begin{equation}\label{eq:SeqPhiPsi}
\Phi^{(b)}(F):={\mu_b^{\otimes 2q}(F^\sharp)^{1/2}\over \mu_b^{\otimes q}(F)}\qquad\text{and}\qquad\Psi_j^{(b)}(F):={\mu_b^{\otimes q}(F_j^*)\over \mu_b^{\otimes q}(F)}\,.
\end{equation}

\begin{theorem}\label{thm:BleiJansonGeneral}
Let $N$ be a standard Gaussian random variable. Then there are constants $0<C_1,C_2<\infty$ only depending on $q$ such that $$d_K(S^{(b)}(F),N)\leq C_1\,\Phi^{(b)}(F)+C_2\,\big(\sup_{j\geq 1}\Psi_j^{(b)}(F)\big)^{1/4}\,.$$
\end{theorem}
\begin{proof}
We notice that $S^{(b)}(F)$ can be written as a discrete multiple stochastic integral $J_q(f)$ of order $q$ with kernel function $$f(i_1,\ldots,i_q)={b_{i_1}\cdots b_{i_q}\over (q!\mu_b^{\otimes q}(F))^{1/2}}\,\1_{\{(i_1,\ldots,i_q)\in F\}}\,.$$ Furthermore, $S^{(b)}(F)$ has unit variance so that the first inequality of Theorem \ref{BerryEsseenMultipleIntegrals} implies that $d_K(S^{(b)}(F),N)$ is bounded from above by $$C\,\max\big\{\max_{r=1,\ldots,q-1}\{\|(f\star_r^r f)\1_{\Delta_{2(q-r)}}\|_{\ell^2(\N)^{\otimes 2(q-r)}},\max_{r=1,\ldots,q}\{\|f\star_r^{r-1}f\|_{\ell^2(\N)^{\otimes(2(q-r)+1)}}\}\big\}$$ with a constant $0<C<\infty$ only depending on $q$. These contraction norms can be computed exactly as in the proof of Theorem 6.4 in \cite{ReiPec}, so that we leave out the details. This gives the result.
\end{proof}

We now specialize Theorem \ref{thm:BleiJansonGeneral} to the set-up discussed in \cite{BleJan}. For this, let $(F_n)_{n\geq 1}$ be a sequence of non-empty subsets of $\N^q$ satisfying the following two properties:
\begin{itemize}
\item[(i)] $F_n\subset\Delta_q\cap\{1,\ldots,n\}^q$ for all $n\in\N$,
\item[(ii)] $F_n$ is symmetric in that $(i_1,\ldots,i_q)\in F_n$ implies that $(i_{\sigma(1)},\ldots,i_{\sigma(q)})\in F_n$ for all permutations $\sigma$ of $\{1,\ldots,q\}$.
\end{itemize}
Further let $b_0=(b_i)_{i\geq 1}$ be such that $b_1=\ldots=b_q=1$ and $b_i=0$ for all $i\geq q+1$. The sequence $(S_n)_{n\geq 1}$ of random variables is then given by $S_n:=S^{(b_0)}(F_n)$ with $S^{(b_0)}(F_n)$ as at \eqref{eq:DefSbF}. Note that the measure $\mu_{b_0}$ reduces to the counting measure restricted to $\{1,\ldots,q\}$ and we write $\Phi(F_n)$ and $\Psi_j(F_n)$ in \eqref{eq:SeqPhiPsi} instead of $\Phi^{(b_0)}(F_n)$ and $\Psi_j^{(b_0)}(F_n)$, respectively. For this setting, Theorem \ref{thm:BleiJansonGeneral} yields the following Berry-Essen bound, which improves Theorem 1.7 in \cite{BleJan} and Theorem 6.4 in \cite{ReiPec}.

\begin{corollary}\label{cor:BleiJanson}
Let $N$ be a standard Gaussian random variable. Then there are constants $0<C_1,C_2<\infty$ only depending on $q$ such that $$d_K(S_n,N)\leq C_1\,\Phi(F_n)+C_2\big(\max_{1\leq j\leq n}\Psi_j(F_n)\big)^{1/4}$$ for all $n\in\N$.
\end{corollary}

A concrete situation to which Corollary \ref{cor:BleiJanson} can be applied concerns fractional Cartesian products. We briefly recall their definition and refer to \cite{Ble} for further details. Fix integers $q\geq 3$ and $m\in\{2,\ldots,q\}$, and let $\{M_1,\ldots,M_q\}$ be a collection of distinct non-empty subsets of $[q]:=\{1,\ldots,q\}$ with the following properties:
\begin{itemize}
\item[(i)] $\{M_1,\ldots,M_q\}$ is a connected cover of $[q]$, i.e., $\bigcup\limits_{i=1}^qM_i=[q]$ and $\{M_1,\ldots,M_q\}$ cannot be partitioned into two disjoint partial covers,
\item[(ii)] each subset has exactly $m$ elements,
\item[(iii)] each index $j\in[q]$ appears in exactly $m$ of the subsets $M_1,\ldots,M_q$.
\end{itemize}
For a subset $M\subset[q]$ and a vector ${\bf y}=(y_1,\ldots,y_q)$ we write $\pi_M{\bf y}:=(y_j:j\in M)$ for the projection of ${\bf y}$ on $M$. Now, let for an integer $n\geq q^m$, $K:=\max\{k\in\N:k\leq n^{1/m}\}$, fix a one-to-one map $\varphi:[K]^m\to[n]$ and define $$F_n^{**}:=\{(\varphi(\pi_{M_1}{\bf y}),\ldots,\varphi(\pi_{M_q}{\bf y})):{\bf y}=(y_1,\ldots,y_q)\in[K]^q\}\subset[n]^q\,.$$ In general, $F_n^{**}$ is not symmetric and may contain diagonal elements. We thus define
$$F_n:=\{(i_1,\ldots,i_q)\in\Delta_q^n:(i_{\sigma(1)},\ldots,i_{\sigma(q)})\in F_n^{**}\cap\Delta_q^n\text{ for some permutation $\sigma$ of $[q]$}\}\,,$$
where $\Delta_q^n$ stands for $\Delta_q\cap\{1,\ldots,n\}^q$. The sequence $(F_n)_{n\geq 1}$ is what is called a fractional Cartesian product in the literature. The name comes from the fact that $F_n$ has (fractional) combinatorial dimension $q/m$, a notion for which we refer to \cite{Ble2}. We take such a fractional Cartertesian product as in \eqref{eq:DefSbF} as input sets for our random variables and write $S_n^{\rm fCP}$ in this case. Then Corollary \ref{cor:BleiJanson} and the computations leading to Proposition 6.6 in \cite{ReiPec} imply the following Berry-Esseen bound.

\begin{corollary}\label{cor:BleiJansonfCP}
Let $N$ be a standard Gaussian random variable. Then there is a constant $0<C<\infty$ such that $$d_K(S_n^{\rm fCP},N)\leq C\, n^{-1/(2m)}$$ for all $n\geq 1$.
\end{corollary}

\begin{remark}
The rate in Corollary \ref{cor:BleiJansonfCP} for the Kolmogorov distance is the same as in Proposition 6.6 in \cite{ReiPec}, where the authors considered a probability metric based on twice differentiable test functions.
\end{remark}

\subsection{Traces of powers of Bernoulli matrices}
As an application of Corollary \ref{cor:MutiIntegrals} we consider the normal approximation of a vector of traces of powers of a Bernoulli random matrix. This task has already been accomplished in \cite{NouPecMatrices} for more general classes of random matrices whose entries are independent and obey certain moment conditions. In this paper the authors used universality results for homogeneous sums established in \cite{NouPecReiInvariance}. Our Corollary \ref{cor:MutiIntegrals} offers a bound for the multivariate normal approximation for the special case of Bernoulli random matrices without resorting to universality results. For related limit theorems dealing with traces of random matrices we refer to \cite{AndZeit,Guionnet} and the references cited therein.\\
As already indicated in \cite{ReiPec} the results of discrete Malliavin calculus and Stein's method extend to Rademacher random variables indexed by arbitrary discrete sets. We make use of this possibility by  considering a doubly indexed collection of i.i.d.\ Rademacher random variables $(X_{ij})_{i,j\in\N}$. The object of interest is the {Bernoulli random matrix}
\[X_n:=\left({X_{ij}\over\sqrt {n}}\right)_{1\leq i,j\leq n}.\]
Let us denote by 
\[\text{trace}(X_n^q):=n^{-q/2}\sum_{i_1,i_2,\dots,i_q=1}^n X_{i_1i_2} X_{i_2i_3}\cdots X_{i_qi_1}\]
 the trace of the $q$th power of $X_n$. 
Recall the definition of the $d_4$-distance in \eqref{eq:d4Definition}.

\begin{theorem}\label{thm:RandMatrices}
Let $d\geq 2$ and $1\leq q_1<\dots<q_d$ be integers. Define the random vector 
\[\textbf{F}_n=\big({\rm trace}(X_n^{q_1})-\E[{\rm trace}(X_n^{q_1})],\dots,{\rm trace}(X_n^{q_d})-\E[{\rm trace}(X_n^{q_d})]\big)\] 
and let $\textbf{N}$ be a centred Gaussian random vector with covariance matrix $\Sigma=(\sigma_{ij})_{i,j=1}^d$ such that $\sigma_{ii}=q_i$ for $1\leq i\leq d$ and $\sigma_{ij}=0$ for $1\leq i\neq j\leq d$. Then
\begin{align}\label{eq:d40Matrices}
d_4(\textbf{F}_n,\textbf{N})\leq C\, n^{-{1/ 4}}
\end{align}
for a constant $0<C<\infty$ depending only on $q_1,q_2,\dots,q_d$ and $d$.
\end{theorem}

\begin{proof}
In order to apply Corollary \ref{cor:MutiIntegrals} we have to express $\text{trace}(X_n^{q_i})$ in terms of elements of some fixed Rademacher chaos. To this end, the following decomposition taken from \cite{NouPecMatrices} is crucial. It holds that

\[\text{trace}(X_n^{q})=n^{-q/2}\!\sum_{(i_1,\dots,i_q)\in D_n^{(q)}}X_{i_1i_2} X_{i_2i_3}\cdots X_{i_qi_1}+n^{-q/2}\!\sum_{(i_1,\dots,i_q)\notin D_n^{(q)}}X_{i_1i_2} X_{i_2i_3}\cdots X_{i_qi_1},\]

where \[D_n^{(q)}=\{(i_1,\dots,i_q)\in \{1,\dots,n\}^q:\, (i_a,i_{a+1})\neq(i_b,i_{b+1}) \text{ for } a\neq b\}\]
with $i_{q+1}:=i_1$.
This separation of the range of summation is necessary due to the fact that multiple integrals are defined only for kernels that are symmetric and vanish on diagonals. One has
\[n^{-q/2}\!\sum_{(i_1,\dots,i_q)\in D_n^{(q)}}X_{i_1i_2} X_{i_2i_3}\cdots X_{i_qi_1}=J_q\big(f^{(q)}_{n}\big),\]
where $f^{(q)}_{n}:=\widetilde{f_{q,n}}$ is the canonical symmetrization of
\[f_{q,n}((a_1,b_1),\dots,(a_q,b_q)):=n^{-q/2}\!\sum_{(i_1,\dots,i_q)\in D_n^{(q)}}\1_{\{i_{1}=a_1,i_{2}=b_1\}} \1_{\{i_{2}=a_2,i_{3}=b_2\}}\cdots\1_{\{i_{q}=a_q,i_{1}=b_q\}}.\]
Let $\textbf{J}_n=(J_1,\dots,J_d)$ denote the random vector $\big(J_{q_1}\big(f^{(q_1)}_{n}\big),\dots,J_{q_d}\big(f^{(q_d)}_{n}\big)\big)$. Then the triangle inequality for the $d_4$-distance implies that
\begin{equation}
d_4\left(\textbf{F}_n,\textbf{N}\right)\leq d_4\left(\textbf{F}_n,\textbf{J}_n\right)+d_4\left(\textbf{J}_n,\textbf{N}\right)\,.\label{eq:d4}
\end{equation}
From \cite[Equation (3.9)]{NouPecMatrices} one has
\begin{equation}
\big|q_i-\E\big[J_{q_i}\big(f^{(q_i)}_{n}\big)^2\big]\big|= O(n^{-1})\label{eq:VarTrace}\,,
\end{equation}
as $n\to\infty$ for all $i=1,\ldots,d$. This has actually been established in \cite{NouPecMatrices} for random matrices with independent Gaussian entries. But it is easily checked that the estimate continues to hold for Rademacher random variables. This together with the isometry \eqref{Isometry property} of discrete multiple stochastic integrals implies the information on the covariance structure of ${\bf F}_n$.

Next, from \cite[Equation (3.6)]{NouPecMatrices} one has
\begin{equation}
\lnormb{2}{2(q-r)}{f^{(q)}_{n}\star_r^rf^{(q)}_{n}}=O(n^{-1/2})\label{eq:normsTraces}
\end{equation}
for all $r=1,\dots,q-1$ in the limit, as $n\to\infty$. In view of \eqref{eq:VarTrace} and \eqref{eq:normsTraces}, Corollary \ref{cor:MutiIntegrals} yields
\begin{equation}
d_4\left(\textbf{J}_n,\textbf{N}\right)= O(n^{-1/4})\,.\label{eq:TraceB1}
\end{equation}

Furthermore, Proposition 4.1 in \cite{NouPecMatrices} states that for all $1\leq i\leq d$,
\begin{align}
&\quad\E[(F_i-J_i)^2]\notag\\
&=\E\Big[\Big(n^{-q_i/2}\!\sum_{(i_1,\dots,i_{q_i})\notin D_n^{(q_i)}}\big(X_{i_1i_2} X_{i_2i_3}\cdots X_{i_{q_i}i_1}-\E[X_{i_1i_2} X_{i_2i_3}\cdots X_{i_{q_i}i_1}]\big)\Big)^2\Big]=O(n^{-1})\,.\label{eq:RemainderBound}
\end{align}
Let $g:\R^d\rightarrow\R$ be an admissible test function for the $d_4$-distance. Then, writing $\|\,\cdot\,\|_{\R^d}$ for the standard euclidean norm in $\R^d$, we find that
\begin{align}
\big|\E g({\bf F}_n)-\E g({\bf J}_n)\big|&\leq M_1(g)\,\E[\|{\bf F}_n-{\bf J}_n\|_{\R^d}]
\leq \E\bigg[\big(\sum_{i=1}^d(F_i-J_i)^2\big)^{1/2}\bigg]\notag\\
&\leq \Big(\sum_{i=1}^d\E[(F_i-J_i)^2]\Big)^{1/2}=O(n^{-1/2})\,,\label{eq:TraceB2}
\end{align}
for a constant $0<C<\infty$ which is independent of $n$. Here, we have used the fact that $M_1(g)\leq 1$, the Cauchy-Schwarz inequality and \eqref{eq:RemainderBound}. The decomposition \eqref{eq:d4} together with \eqref{eq:TraceB1} and \eqref{eq:TraceB2} proves the assertion.
\end{proof}

\begin{remark}
The rate of convergence obtained in Theorem \ref{thm:RandMatrices} by means of the Malliavin-Stein method for Rademacher sequences is of the same order of magnitude as the rate in \cite{NouPecMatrices}, which is based on the universality results in \cite{NouPecReiInvariance}. The only difference is that for technical reasons we had to assume a slightly higher degree of differentiability for the test functions in the probability metric.
\end{remark}

\begin{remark}
For the proof of Theorem \ref{thm:RandMatrices} we have used the second inequality in Corollary \ref{cor:MutiIntegrals}, which involves square-roots of contraction norms. One might hope to improve \eqref{eq:d40Matrices} by using the first inequality in Corollary \ref{cor:MutiIntegrals}. For this, one would need to extend \eqref{eq:normsTraces} to an estimate for norms of mixed contractions. We  expect that this requires considerable effort since already the proof of \eqref{eq:normsTraces} relies on involved combinatorial decompositions and identities.
\end{remark}


\bibliography{Rademacher}

\end{document}